\documentclass[ 
  api,
  pre,
  superscriptaddress,
  floatfix,
  nofootinbib,
  longbibliography,
  notitlepage,
  onecolumn,
  showkeys
]{revtex4-2}
\usepackage{xcolor}
\definecolor{forest}{RGB}{66,111,68}
\usepackage[colorlinks=true, allcolors=forest]{hyperref}

\usepackage{dcolumn}

\usepackage[english]{babel}

\usepackage[utf8]{inputenc}
\usepackage[T1]{fontenc}
\usepackage{etoolbox}

\usepackage{amsfonts}
\usepackage{mathtools}
\usepackage{graphicx}
\usepackage{subfig}
\usepackage{csquotes}
\usepackage{multirow}
\usepackage{fnpct}      
\usepackage{booktabs}   
\usepackage{paralist}   
\usepackage{fnpct}      
\usepackage{enumitem}   
\usepackage{comment}
\usepackage{float}
\usepackage{array}
\usepackage{tabularray}
\UseTblrLibrary{booktabs}

\usepackage{orcidlink}
\usepackage{ragged2e}

\DeclareCaptionJustification{justified}{\justifying}
\usepackage[ruled,vlined]{algorithm2e}

\usepackage[noend]{algpseudocode}

\newcommand{\MB}[1]{\textcolor{mb}{\textbf{MB}: #1}}   

\newcommand{\defeq}{\vcentcolon=}

\definecolor{greyblue}{RGB}{45,115,196}
\definecolor{date}{RGB}{94,129,172}
\definecolor{node}{RGB}{67,76,94}
\definecolor{path}{RGB}{76,86,106}
\definecolor{mob}{RGB}{121, 158, 89}
\definecolor{start}{RGB}{129,161,193}
\definecolor{submission}{RGB}{191,97,106}
\definecolor{gender}{RGB}{191, 97, 96}
\definecolor{wellow}{RGB}{221, 167, 59}
\definecolor{diversity}{RGB}{191,97,106}
\definecolor{mig}{RGB}{25,61,108}
\definecolor{presentation}{RGB}{45,115,196}
\definecolor{beamer}{RGB}{35,55,59}
\definecolor{cobal}{rgb}{0.38,0.48,0.55}
\definecolor{bordeaux}{rgb}{0.659,0.063,0.063}
\definecolor{nodecolor}{RGB}{253, 240, 213}

\definecolor{hp2}{RGB}{96, 0, 0}
\definecolor{hp1}{RGB}{193, 18, 31}
\definecolor{hp3}{RGB}{247, 178, 103}
\definecolor{hp5}{RGB}{0, 48, 73}
\definecolor{hp4}{RGB}{102, 155, 188}
\definecolor{hp1line}{RGB}{102, 7, 8}
\definecolor{hp4line}{RGB}{0, 53, 102}

\definecolor{c1}{HTML}{3c6e71}
\definecolor{c2}{HTML}{a8dadc}
\definecolor{c3}{HTML}{e9c46a}
\definecolor{c4}{HTML}{a3a6c2}
\definecolor{c5}{HTML}{bc4749}
\definecolor{c6}{HTML}{adc178}
\definecolor{c7}{HTML}{efb8a9}

\definecolor{edgecolor}{RGB}{253, 240, 213}
\definecolor{cite}{HTML}{52796f}

\definecolor{col2}{HTML}{8d99ae}
\definecolor{col1}{HTML}{bc4749}
\definecolor{col7}{RGB}{247, 178, 103}
\definecolor{col4}{HTML}{1d2d44}
\definecolor{col5}{HTML}{adc178}
\definecolor{col6}{HTML}{ffe5ec}
\definecolor{col3}{HTML}{FFF3B0}
\definecolor{mb}{HTML}{023e8a}
\definecolor{interval}{HTML}{ef233c}
\definecolor{deep}{HTML}{1d3557}

\usepackage{wasysym}
\usepackage{graphicx}
\usepackage{stackengine}

\newcommand\restr[2]{{
  \left.\kern-\nulldelimiterspace 
  #1 
  \littletaller 
  \right|_{#2} 
  }}
\newcommand{\littletaller}{\mathchoice{\vphantom{\big|}}{}{}{}}

\usepackage{dsfont}
\usepackage{amssymb}
\usepackage{amsmath}
\usepackage{latexsym}
\usepackage{amsthm}

\usepackage{tikz}
\usetikzlibrary{positioning}
\usetikzlibrary{patterns}
\usetikzlibrary{intersections}
\usetikzlibrary{decorations.pathreplacing,positioning,calc,tikzmark,shapes.geometric, arrows, arrows.meta, plotmarks, shapes, bending}
\usetikzlibrary{shapes.geometric}
\usepackage{fontawesome5} 

\usepackage[capitalise,noabbrev]{cleveref}

\theoremstyle{plain} 
\newtheorem{thm}{Theorem}[section] 
 
\newtheorem{lem}[thm]{Lemma}

\theoremstyle{definition} 
\newtheorem{defn}[thm]{Definition} 
\theoremstyle{remark} 
\newtheorem{oss}[thm]{Remark}

\begin{document}  

\title{Fibration Symmetries and Cluster Synchronization in Multi-Body Systems}

\author{\orcidlink{0000-0003-4372-9734} Margherita Bertè}
 \email{margherita.berte@imtlucca.it}
 \affiliation{IMT School for Advanced Studies, P.zza San Francesco 19, 55100 Lucca (Italy).}
\author{\orcidlink{0000-0001-9154-1758} Tommaso Gili}
 \email{tommaso.gili@imtlucca.it}
 \affiliation{IMT School for Advanced Studies, P.zza San Francesco 19, 55100 Lucca (Italy).}

\date{\today}

\begin{abstract}
Structural symmetries influence synchronization patterns in complex systems. In particular, for graphs is established the interplay between clusters of synchronization for dynamical processes and classes of equivalent nodes that have isomorphic input trees. Recent works generalized and proved for directed hypergraphs that balanced partitions characterize robust synchrony invariant under all admissible dynamics.
In this work, we address undirected hypergraphs and identify node partitions based on equivalent incidence relations in a hypergraph, using its bipartite graph representations. We also show that for the Kuramoto model with higher-order interactions and frustration parameters, clusters of identical oscillators with the same initial conditions that synchronize coincide with the fibration partition for almost all choices of the parameters. We provide algorithms, code implementations, and visualization methods to compute local input equivalence and analyse its relation to synchronization behavior. Simulations with real-world and synthetic topologies illustrate how structural properties and representational choices affect cluster formation. We discuss topology adjustments for target cluster configurations under noise and incomplete information, identifying directions for future investigations.
\end{abstract}


\maketitle

\section{Introduction}
The study of symmetry in complex systems has been central to uncovering principles governing collective dynamics, particularly synchronization. Classical approaches, grounded in group theory and graph automorphisms, have revealed how structural features shape the evolution of connected systems. Yet in many biological and technological networks, global symmetries are rare. This has led to a shift toward local symmetry frameworks, utilizing groupoid theory rather than traditional group theory. Groupoids, introduced by Brandt in 1926, have since served various mathematical contexts~\cite{brown_groups_1987}, offering a generalized framework that extends and relaxes group properties. In our context, we focus on groupoids made by graph fibrations symmetries, which partition nodes based on equivalence in their input flows and have proven effective in identifying cluster synchronization patterns.
Indeed, graph fibrations, a concept rooted in category theory~\cite{grothendieck_technique_1960}, provide a useful framework to uncover synchronization patterns in complex networks and their relation with the topology~\cite{boldi_fibrations_2002, deville_modular_2013, nijholt_graph_2016}.
While extensive research has focused on fibration symmetries in graphs, the growing interest in systems with higher-order interactions (HOI) \cite{battiston_networks_2020, bick_what_2023}  requires further development in this area. In such systems, interactions are not limited to pairs but involve groups of nodes, posing new challenges for understanding how structure governs dynamics. We position our work within this landscape by focusing on the generalization of fibration symmetries to HOI systems.

\subsection*{Related works}

Graph fibrations and symmetry-based methods have been widely explored across different domains to understand synchronization and collective dynamics, from distributed computing~\cite{boldi_fibrations_2002}, gene regulatory systems~\cite{leifer_predicting_2021}, to the neuronal network~\cite{morone_symmetry_2019}. Recent advancements in network theory have shed light on the deep connection between network structure and function, particularly through the lens of graph fibrations and symmetry groups~\cite{boldi_fibrations_2002, deville_modular_2013, golubitsky_nonlinear_2006, nijholt_graph_2016}.

In the context of biological networks, Morone et al.~\cite{morone_symmetry_2019} applied symmetry group factorization to the neural connectome of \textit{Caenorhabditis elegans}, showing a hierarchical structure that matches broad functional categories. 
Further, Leifer et al.~\cite{leifer_predicting_2021} applied fibration symmetries to predict synchronized gene coexpression patterns in bacterial gene regulatory networks. Their findings suggest that network structure alone can provide valuable insights into gene synchronization, potentially reflecting evolutionary pressures that have shaped gene input functions to realize coexpression.
The application of these theoretical frameworks to human brain networks is beginning to yield insights into cognitive function. Makse et al.~\cite{makse_fibration_2024} applied fibration symmetry analysis to brain networks involved in language processing, revealing how symmetry breaking supports functional transitions between resting state and language engagement. That work suggests that local network symmetry in the brain may play a crucial role in determining coherent function, offering a new perspective on the structure-function relationship in neuroscience.

In all cases where global symmetries do not apply, as in the biological realm, the shift of focus from global to local patterns and their impact on the entire system has gained more attention. With this perspective, cluster synchronization refers to the phenomenon where subsets of nodes in a network synchronize their dynamics while remaining out of sync with nodes in other clusters. This behavior is often associated with structural symmetries or equitable partitions in the network, which constrain the dynamical evolution and lead to partial synchronization patterns. Different works focus on the concept of cluster (or partial) synchronization. 
On the one hand, providing computational perspectives through algorithms to partition the nodes~\cite{kamei_computation_2013} or to optimize the analysis of stability for local dynamical patterns~\cite{zhang_unified_2021}. On the other hand, providing a case study or a unified framework and theory related to dynamical processes and node partition~\cite{ gambuzza_stability_2021, pecora_cluster_2014,salova_cluster_2022}.

The Kuramoto model describes coupled phase oscillators and typically leads to global synchronization under strong coupling when natural frequencies are identical. To capture more complex phenomena, such as partial or inhibited synchronization, researchers have extended the model by introducing phase-lag (frustration) terms. Nicosia et al.~\cite{nicosia_remote_2013} linked these frustrated dynamics to remote synchronization mediated by network symmetries, a relationship further explored in the comprehensive overview on cluster synchronization~\cite{schaub_graph_2016}. Kirkland et al.~\cite{kirkland_alphakuramoto_2013} introduced $\alpha$-Kuramoto partitions to study synchronization patterns constrained by a phase lag, while Brede~\cite{brede_frustration_2016} explored heterogeneous phase shifts via local frustration parameters. 

Beyond pairwise interactions, higher-order interactions generalize the notion of connectivity to groups of nodes, posing additional challenges for analysing partial synchronization. Recent studies on hypergraphs or simplicial complexes often rely also on equitable partitions to understand these patterns~\cite{aguiar_network_2023, salova_cluster_2022}. 
Multi-body generalizations of the Kuramoto model, incorporating frustration or phase-lag parameters, have been proposed to capture more complex dynamical behaviors. However, existing generalizations have primarily focused on all-to-all configurations. Notable contributions include Dutta et al.'s asymmetric Kuramoto model for hypergraphs~\cite{dutta_impact_2023}, which generalizes frustration to higher-order interactions, Moyal et al.'s work specifically addressing frustration in hypergraphs limited to three-body interactions~\cite{moyal_rotating_2024}, and León et al.~\cite{leon_higherorder_2024}, who showed that higher-order interactions with nonzero phase-lag in globally coupled ensembles can produce anomalous transitions to synchrony. Further, studies on Kuramoto models with bi and triadic interactions highlight that the impact of higher-order interactions on total synchronization can nontrivially alter both the linear and basin stability of Kuramoto dynamics~\cite{zhang_deeper_2024} and that depend on the chosen representation, enhancing or suppressing collective synchronization in hypergraphs versus simplicial  complexes~\cite{zhang_higherorder_2023}.
This point aligns with recent findings by Salova et al.~\cite{salova_analyzing_2021}, who emphasized the impact of representational choices on synchronization dynamics. 

In parallel, for multi-body interactions, the study of local symmetries has provided a structural perspective on cluster synchronization for directed hypergraphs~\cite{aguiar_network_2023, vondergracht_hypernetworks_2023}. In particular, Von der Gracht et al.\ formally extended the notion of fibrations to systems with higher-order interactions and demonstrated that balanced partitions induced by fibrations identify robust synchrony spaces for all admissible dynamics.
Notably, the comprehensive work~\cite{makse_symmetries_2025} about fibrations symmetries dedicates a chapter to fibrations for higher-order interactions. Our objects of study are undirected hypergraphs, while for the particular case of simplicial complexes, we refer to \cite{nijholt_dynamical_2022}, in which Nijholt et al. studied the equivalent concept of balanced coloring, considering also the orientation of the simplices.
We also aim to prove how local structural equivalences can relate to partial synchronisation in a Kuramoto model, even between nodes that are not directly connected. This emphasizes the relevance of remote synchronization phenomena, which can arise independently of global automorphism symmetries. To this end, we focus on a nonlinear dynamical model, which allows us to capture complex behaviors that more abstract linear approaches might miss, as highlighted in~\cite{vondergracht_hypernetworks_2023}. We focus on fibration symmetries, as they offer a principled way to partition nodes based on identical input structures. This choice differentiates our approach from other similar studies that refer in general to symmetries or equitable partitions, \cite{aguiar_network_2023,salova_cluster_2022}. Indeed, our notion, which corresponds to the coarsest equitable partition or minimal balanced coloring, is central to our analysis not only for its connection to dynamics but also because it captures intrinsic structural regularities that are relevant on their own. 
Building on the theoretical results of~\cite{aguiar_network_2023, vondergracht_hypernetworks_2023} for directed hypergraphs, who proved that balanced partitions characterize robust synchrony invariant under all admissible dynamics, we restrict attention to the higher-order Kuramoto model with frustration, showing that clusters of synchrony coincide with the fibration partition under our model assumptions.
Our approach intersects with the recent work~\cite{sanchez-garcia_equitability_2025} in the use of local symmetries and equitable relations to explain collective dynamics. While Kovalenko et al. establish general conditions for equitable synchronization in multiplex and higher-order networks, we provide computational tools and empirical validation for the specific case of higher-order Kuramoto dynamics.
From a technical viewpoint, the linear independence assumptions on the coupling functions introduced in~\cite{sanchez-garcia_equitability_2025} play the same role as our genericity argument. In their framework, linear independence of the coupling families is imposed as a hypothesis to rule out pathological cancellations between different coupling contributions, which would allow synchrony patterns not aligned with the structurally equitable partitions. 
In our setting, such cancellations correspond to a measure–zero subset of the parameter space, apart from which, Kuramoto synchrony partitions necessarily coincide with the fibration partition.

This work refines the understanding of fibration symmetries in systems with higher-order interactions by addressing how to compute fibres in undirected hypergraphs effectively. Focusing on the higher-order Kuramoto model with frustration, a well-established setting in physics, we show that fibre partitions correspond to invariant synchronization clusters under specific structural conditions.
Our analysis establishes that, starting from the agnostic case of identical oscillators with identical initial conditions, synchronous states not aligned with the fibration decomposition are instantaneously destabilized by mismatches in the input tree of the nodes.
Alongside the theoretical discussion, we provide algorithmic insights and practical procedures to implement fibre detection and symmetry-preserving topology modifications. These contributions bridge abstract formulations and computational practice, offering concrete tools for analysing synchronization phenomena in complex higher-order networks.
Throughout the paper, we provide visualizations and examples to illustrate key concepts, maintaining an operative perspective through algorithmic implementations and references to real-world network topologies.
By integrating theoretical insights with empirical evidence, our research provides a conceptual framework and methodologies for examining HOI systems via fibration symmetries.
Finally, in the concluding section, we discuss the practical challenges encountered with hypergraph fibrations in symmetry-driven network modification~\cite{boldi_quasifibrations_2022, gambuzza_controlling_2021, leifer_symmetrydriven_2022}.
We examine hypergraph sparsification methods that preserve fibre structures while removing redundant hyperedges to reduce computational complexity and improve interpretability. 
Further, we investigate topology modification techniques to reconstruct hypergraph structures when observed synchronization patterns suggest that the available topology may be incomplete or corrupted due to measurement limitations or data loss.

\section{Results}

\subsection{Graph fibration symmetries}

This subsection introduces the formal mathematical definitions of graph fibrations. By fixing the notation and presenting common basic definitions, the aim is to provide a unified framework to facilitate the interpretation of the concepts discussed throughout this work and as a reference point for extending fibration symmetries to complex systems with higher-order interactions.
 
\begin{defn}[Directed graph]
    A directed graph $G$ is a tuple $G = (N, E, s, t)$ where $N$ is the set of nodes and $E \subseteq N \times N$ is the set of edges. The source and target maps, respectively $s, t: E \rightarrow N$, associate each edge with its tail or head, namely the first or second nodes in the ordered pair of nodes constituting the edge.
\end{defn}

\begin{defn}[Undirected graph]
    An undirected graph $G = (N, E, s, t)$  is a directed graph such that $\forall e \in E, \, e = (s(e), t(e)) \Rightarrow \exists f \in E \text{ s.t. } f = (t(e), s(e))$.
\end{defn}

\begin{oss}
    An undirected graph can be characterized by specifying only the set of nodes and edges. Hence, we will omit the source and target maps.
\end{oss}

We now proceed to define fibration symmetries for directed graphs, adopting the notation by Boldi et al.~\cite{boldi_fibrations_2002} who organized and presented the fibration symmetries theory for graphs.

\begin{defn}\label{def: graph_fibration}
    Let $G = (N_G, E_G, s_G, t_G), B = (N_B, E_B, s_B, t_B)$ be directed graphs. 
    A \textbf{graph homomorphism} $\xi: G \rightarrow B$ is given by a pair of functions $\xi_N: N_G \rightarrow N_B$ and $\xi_E: E_G \rightarrow E_B$ commuting with the source and target maps.

    A \textbf{fibration} between $G$ and $B$ is a homomorphism $\varphi: G \rightarrow B$ such that $\forall  a \in E_B, \forall x \in N_G$ satisfying $\varphi(x) = t_B(a)$ there is a unique arc $\tilde{a}^{x} \in E_G$ such that $\varphi(\tilde{a}^{x}) = a$ and $t_G(\tilde{a}^{x}) = x$.
    In literature, we refer to $G$ as the total graph and to $B$ as the base graph.
    This property is called the lifting property.

    Nodes mapped in the same node in the base graph $B$ are a \textbf{fibre}.
    A trivial fibre is a singleton.
    
    Fibrations that collapse all fibres in one node in the base graph are called minimal.
\end{defn}

Throughout the paper, we will consider in particular surjective minimal fibrations, referred to as \textbf{fibration symmetry}, following the definition in \cite {leifer_symmetryinspired_2022}. 
The concept of input sets and input trees, which encode the path of incoming edges to a node, was introduced by Aldis \cite{aldis_polynomial_2008} and by Boldi et al. \cite{boldi_fibrations_2002}. They employed these notions to prove that finding nodes having isomorphic input trees is equivalent to the existence of graph fibrations.
Further, the proof of the equivalence between fibres of minimal surjective fibration and minimal balanced coloring is stated by Aldis \cite{aldis_polynomial_2008}.
This maximal grouping ensures that nodes are partitioned only when their input structures genuinely differ, avoiding artificial fragmentation of structurally equivalent nodes. Our approach is deliberately topology-centric: nodes are classified by their incidence relations (i.e., by their input trees of colored input set), enabling us to isolate the role of network structure in determining synchronization.

\begin{oss}
 \cref{def: graph_fibration} can be extended to undirected graphs by treating them as bidirected graphs, meaning that each undirected edge is regarded as two directed edges that connect the two possible configurations of source and target vertices. 
\end{oss}

\begin{figure}[ht]
    \centering
    \scalebox{1.1}{\begin{tikzpicture}[squarednode/.style={rectangle, scale=0.1}]
    
                    \tikzset{markerpieno/.style={circle, draw, fill, scale=0.5, color = node}}
                    \tikzset{marker/.style={circle, draw, scale=0.5}}


                    \node[squarednode] (ghost6) at (2.2,-1.8) {};
                    \node [marker, above left=of ghost6, node distance=2in, color=node, fill=mig!50] (G1) {};
                    \node [marker, left=of G1, node distance=2in, color=node, fill=mig!50] (G2) {};
                    \node [marker, below=of G2, node distance=2in, color=node] (G3) {};
                    \node [marker, right=of G3, node distance=2in, color=node, fill=col5!70] (G4) {};
                    
                    \path[-{Straight Barb[length=2.5pt,width=2.5pt]}, color = node, line width=0.2mm] (G4) edge node[above] {} (G2);
                    \path[-{Straight Barb[length=2.5pt,width=2.5pt]}, color = node, line width=0.2mm] (G4) edge node[above] {} (G1);
                    \path[-{Straight Barb[length=2.5pt,width=2.5pt]}, color = hp1, line width=0.2mm] (G3) edge node[above] {} (G1);
                    \path[-{Straight Barb[length=2.5pt,width=2.5pt]}, color = hp1, line width=0.2mm] (G2) edge node[below = 1.4cm, color = black] {G} (G1);
                    \path[->, >={Straight Barb[length=2.5pt,width=2.5pt, flex']}, color = hp1, line width=0.2mm] (G2) edge [out=90, in=180, loop] node[above] {} (G2);
                    \path[-{Straight Barb[length=2.5pt,width=2.5pt]}, color = hp1, bend left=25, line width=0.2mm] (G2) edge node[above] {} (G3);
                    \path[-{Straight Barb[length=2.5pt,width=2.5pt]}, color = hp1, bend left=25, line width=0.2mm] (G3) edge node[above] {} (G2);

                    \node[squarednode, above left=0.5cm and 0.7cm of ghost6] (ghost7) {};
                    \node[squarednode, right=1cm of ghost7] (ghost8) {};
                    
                    \path[-{Straight Barb[length=2.5pt,width=2.5pt]}] (ghost7) edge node[above] {$\varphi$} node[below=1.25cm] {}  (ghost8);

                    \node [marker, right=2.2cm of G1, node distance=2in, color=node, fill=mig!50] (B1) {};
                    \node [marker, below left= 1cm and 0.5cm of B1, node distance=2in] (B2) {};
                    \node [marker, below right= 1cm and 0.5cm of B1, node distance=2in, color=node, fill=col5!70] (B3) {};

                    \path[-{Straight Barb[length=2.5pt,width=2.5pt]}, color = node, line width=0.2mm] (B3) edge node[above] {} (B1);
                    \path[-{Straight Barb[length=2.5pt,width=2.5pt]}, color = hp1, bend left=25, line width=0.2mm] (B2) edge node[above] {} (B1);
                    \path[-{Straight Barb[length=2.5pt,width=2.5pt]}, color = hp1, bend left=25, line width=0.2mm] (B1) edge node[above] {} (B2);
                    \path[->, >={Straight Barb[length=2.5pt,width=2.5pt, flex']}, color = hp1, line width=0.2mm] (B1) edge [color = hp1, line width=0.2mm, out=45, in=135, loop] node [below = 1.75cm, color = black] {B} (B1);


\end{tikzpicture}}    
    \caption{\textit{Surjective graph fibration.} Example of graph fibration $\varphi: G \rightarrow B$. Nodes with the same color belong to the same fibre in the total graph $G$ and are mapped by  $\varphi$ to the unique node with the same color in the base graph $B$, which is the quotient graph.}\label{fig: graph_fibration}
\end{figure} 

Makse et al.~\cite{makse_symmetries_2025} provide a detailed discussion of the different notions of graph quotients that arise when one takes a node partition as an equivalence relation. The definition is subtle, since different conventions lead to different objects. In our setting, we require the quotient to coincide with the base graph in the sense of fibration symmetry. In particular, the construction must guarantee the lifting property.

\begin{defn}[Quotient graph]
Let $G = (N, E)$ be a directed graph and let $\sim$ be the equivalence relation on $N$ induced by a fibration symmetry, i.e., $u \sim v$ whenever $u$ and $v$ belong to the same fibre.
For $v \in N$, denote by $[v]$ the fibre containing $v$.

The quotient graph $G/\sim \,\defeq (\bar N, \bar E)$ is defined as follows:
\begin{itemize}
    \item The node set is $\bar N = \{ [v] \mid v \in N \}$.
    \item For an ordered pair of fibres $[u], [v] \in \bar N$, the number of directed edges between them in $\bar E$ equals the number of edges in $G$ between a vertex of $[u]$ and a vertex of $[v]$.
    \item If $u$ and $v$ lie in the same fibre, loops in $\bar E$ record the edges of $G$ internal to that fibre.
\end{itemize}
\end{defn}

Thus, the quotient graph is obtained by collapsing each fibre to a single node while preserving the connectivity pattern between fibres. In \cref{fig: graph_fibration}, we depict a representation where nodes in the same fibre (with the same color) are collapsed into a unique node in the quotient graph on the right (also named base graph). We can observe that the fibration symmetry $\varphi$ maps each node to a quotient node whose incoming edges originate from nodes of the same fibres as in the total graph.

\subsection{Higher-order interactions}\label{sec: cat_eq}
This subsection sets the notations needed to model higher-order interactions in complex systems. We recall the definitions of hypergraphs and their associated bipartite representations. We then emphasize how the incidence bipartite graph provides a natural framework for identifying fibres in multi-body interactions. Finally, we introduce the notions of fibres and fibration symmetries in the context of hypergraphs.

\begin{defn}[Directed hypergraph]\label{def: dir_hypg}
    A directed hypergraph $H = (N, E)$ consists of a node set $N$ and a set of directed hyperedges $E$. A directed hyperedge $e \in E$ is a pair $(T(e), H(e))$, 
    where the tail $T(e)$ is a non-empty multiset of elements of $N$ 
    and the head $H(e)$ is a non-empty subset of $N$.
\end{defn}

\begin{oss}
    In our setting, we do not require an order among the nodes in the tail or  or head of each hyperedge, differently from \cite{vondergracht_hypernetworks_2023}.
\end{oss}
As a particular case of directed hypergraphs, we can consider undirected hypergraphs (hereafter referred to as hypergraphs). They will be the focus of our work.

\begin{defn}[Hypergraph]\label{def: hypg}
    An undirected hypergraph $H = (N, E)$ consists of a node set $N$ and a hyperedge set $E$, where each hyperedge is a non-empty multiset of nodes from $N$.
    The cardinality of a hyperedge is called its order.
    The rank $rk(H)$ of a hypergraph is the maximum of the orders of its hyperedges, $rk(H) = \max\limits_{e \in E} \lvert e \rvert$.
\end{defn}

\begin{oss}    
    Because hyperedges are multisets, a node may appear more than once within the same hyperedge.
    While this is not the standard convention in hypergraph theory, it is a natural choice when working with quotient constructions, and it has already been adopted in related contexts (e.g.~\cite{aguiar_network_2023, vondergracht_hypernetworks_2023}).
    This convention will be important here to guarantee that the quotient of a hypergraph remains a hypergraph under the current definition (refer to the hypergraph on the right in \cref{fig: hypergraph_quotient}).
\end{oss}

\begin{oss}
      An undirected hypergraph can be considered as a directed hypergraph with all possible combinations of source and target sets in each hyperedge, as nicely depicted in Figure 1 of \cite{gallo_synchronization_2022}.
\end{oss}

\begin{oss}    
    Observe that if $rk(H)=2$, then $H$ is an undirected graph.
\end{oss}

\begin{defn}[Bipartite digraph]\label{def: bip_graph}
    A directed bipartite graph (or bipartite digraph) is a graph $G = (N_1 \sqcup N_2, E)$ where the nodes can be partitioned into two disjoint subsets (named layers) such that $E \subseteq N_1 \times N_2$. It means that the edges cannot contain nodes belonging to the same layer.
\end{defn}

\paragraph*{Incidence graph representation}

A hypergraph $H = (N, E)$ can be represented via its incidence bipartite digraph $B_H = (N \sqcup E, F)$, where $F = \{(n, e) \in N \times E \,\vert\, n \in e\}$, also called factor graph \cite{makse_symmetries_2025}. In \cite{boldi_emergence_2023}, the author proves a categorical equivalence between directed hypergraphs and bipartite RB-graphs.
With the same perspective, in \cref{sec: appendix_cat} we extend the demonstration and propose the categorical equivalence of the undirected hypergraphs and bidirectional bipartite graphs. 
This relationship is naturally intuitive, as the unique incidence matrix (apart from the order) of a hypergraph coincides with the unique bi-adjacency matrix of the bipartite graph having as nodes in the first layer the hypergraph nodes and as nodes in the second layer the hyperedges.
Thus, a hypergraph can be handled by its incidence bipartite representation, and algorithms designed to partition graph nodes with the same input trees can be applied to hypergraph nodes by regarding them as belonging to an inhomogeneous directed graph.

\begin{oss}
 The bipartite graph must not contain any isolated node since disconnected nodes in the second layer would represent empty hyperedges, which fall outside the formal definition of a hypergraph.
\end{oss}

\subsection{Fibration symmetries for hypergraphs}

By representing the hypergraph as its incidence bipartite digraph, we can treat it as an inhomogeneous graph and thus apply fibre-computing algorithms. In this perspective, nodes corresponding to original vertices and nodes corresponding to hyperedges are explicitly distinguished from the outset by assigning them different initial types (or colors). The algorithm then partitions both sets into fibres, producing a coloring of the original nodes as well as of the hyperedges (see \cref{fig: cat_eq}).
The fibres in a hypergraph $H$ are defined as the nodes that belong to identical fibres in the associated incidence bipartite graph $B_H$ when $B_H$ is treated as an inhomogeneous graph.
Fibration symmetries for HOI are fibrations symmetries for the inhomogeneous graph representation of the incidence bipartite graph of the hypergraph, preserving the colors of the original nodes and hyperedges.

\begin{figure}[ht]
    \centering
    \scalebox{.9}{\input{Tikz/equivalence_bip_hyg_thin}}
    \caption{\textit{Workflow to find fibres in hypergraphs.} Starting from an undirected hypergraph, we can associate its incidence bipartite projection. Considering it as an inhomogeneous graph, we can apply standard graph algorithms, as the version of~\cite{kamei_computation_2013} revised by~\cite{morone_fibration_2020}, to obtain the hypergraph node partition, here depicted with different colors.}\label{fig: cat_eq}
\end{figure} 

For hypergraphs, the quotient construction is performed on the incidence bipartite graph, and the resulting quotient can then be translated back into the hypergraph setting. A key point is that preserving multi-body interactions may require allowing repeated occurrences of quotient nodes within the same hyperedge. \cref{fig: hypergraph_quotient} illustrates a fibration symmetry on a hypergraph. Nodes on the left are mapped to nodes of the same shape and color on the right. In the quotient on the right, some nodes appear multiple times in a single hyperedge, highlighting the role of multisets in the definition of a hypergraph.
It is also worth noting that, as for quotient in graphs, even if the original hypergraph is undirected, its quotient need not be. Indeed, when two distinct nodes of the same fibre are connected to a common node in the original hypergraph, collapsing them in the quotient may introduce a non-reciprocated edge, thereby breaking undirectedness.

\begin{figure}[!ht]
    \centering  
    \scalebox{0.8}{\begin{tikzpicture}[scale=0.8, >=latex]
    \node (hpg) at (0,0)
        {
\begin{tikzpicture}[scale=0.8]
                \tikzstyle{point}=[circle,draw=black,fill=black,inner sep=0pt,minimum width=4pt,minimum height=4pt]        
                \node (v1) at (-0.7,0.2) {};
                \node (v2) at (0.7,0.2) {};
                \node (v3) at (0,-1) {};
                \node (v4) at (0,-4) {};
                \node (v5) at (-0.7,-5.2) {};
                \node (v6) at (0.7,-5.2) {};
                \node (v7) at (0,-2.5) {};
                \node (v8) at (1.4,-2.5) {};
                \node (v9) at (2.6,-1.8) {};
                \node (v0) at (2.6,-3.2) {};

                \begin{scope}[fill opacity=0.45]

                    \filldraw[black, fill=edgecolor, rounded corners, line width=0.1mm] ([shift={ (-0.35,0.25)}] v1.west) -- ([shift={ (0.35,0.25)}] v2.east) -- ([shift={ (0,-0.35)}] v3.south) -- cycle;

                    \filldraw[black, fill=edgecolor, rounded corners,line width=0.1mm] ([shift={ (-0.35,0)}] v8.west) -- ([shift={ (0.15,0.44)}] v9.east) -- ([shift={ (0.15,-0.44)}] v0.east) -- cycle;
                
                \end{scope}
                    
                \begin{scope}[fill opacity=0.7]

                    \node[circle, fill=c1, draw = black, scale=0.7] (v1) at (-0.7,0.2) {};
                    \node[circle, fill=c1, draw = black, scale=0.7] (v2) at (0.7,0.2) {};
                    \node[regular polygon, regular polygon sides=3, rotate=180, fill=c6, draw = black, scale=0.6] (v3) at (0,-1) {};
                    \node[star, fill=c3, draw = black, scale=0.6] (v4) at (0,-4) {};
                    \node[regular polygon, regular polygon sides=3, fill=c5, draw = black, scale=0.6pt] (v5) at (-0.7,-5.2) {};
                    \node[regular polygon, regular polygon sides=3, fill=c5, draw = black, scale=0.6pt] (v6) at (0.7,-5.2) {};
                    \node[rectangle, fill=c4, draw = black, scale=0.8] (v7) at (0,-2.5) {};
                    \node[regular polygon, regular polygon sides=5, fill=c2, draw = black, scale=0.6pt] (v8) at (1.4,-2.5) {};
                    \node[diamond, fill=c7, draw = black, scale=0.6] (v9) at (2.6,-1.8) {};
                    \node[diamond, fill=c7, draw = black, scale=0.6] (v0) at (2.6,-3.2) {};

                \end{scope}

                \draw[black, line width=0.1mm, thin] (v3) -- (v7);
                \draw[black, line width=0.1mm, thin] (v4) -- (v5);                      
                \draw[black, line width=0.1mm, thin] (v4) -- (v6);               
                \draw[black, line width=0.1mm, thin] (v5) -- (v6);              
                \draw[black, line width=0.1mm, thin] (v7) -- (v4);              
                \draw[black, line width=0.1mm, thin] (v9) -- (v0);            
                \draw[black, line width=0.1mm, thin] (v7) -- (v8);


    \end{tikzpicture}
  };

    \node (fib) at (5,-0.5) {
        \begin{tikzpicture}[scale=0.8]
            \node (ghost7) at (0,0) {};
            \node[right=1cm of ghost7] (ghost8) {};
            
            \path[-{Straight Barb[length=2.5pt,width=2.5pt]}] (ghost7) edge node[above] {$\varphi$} node[below=1.25cm] {}  (ghost8);
        \end{tikzpicture}
     };

     \node (quot) at (10,0)
        {\begin{tikzpicture}[scale=0.8]
                \tikzstyle{point}=[circle,draw=black,fill=black,inner sep=0pt,minimum width=4pt,minimum height=4pt]        
                \node (v1) at (-0.7,0.2) {};
                \node (v2) at (0.7,0.2) {};
                \node (v3) at (0,-1) {};
                \node (v4) at (0,-4) {};
                \node (v5) at (-0.7,-5.2) {};
                \node (v6) at (0.7,-5.2) {};
                \node (v7) at (0,-2.5) {};
                \node (v8) at (1.4,-2.5) {};
                \node (v9) at (2.6,-1.8) {};
                \node (v0) at (2.6,-3.2) {};
                \node (v09) at (2.5,-2.5) {};
                \node (v00) at (2.6,-2.5) {};

                \begin{scope}[fill opacity=0.45]

                    \filldraw[black, fill=edgecolor, rounded corners, line width=0.1mm] ([shift={ (-0.35,0.25)}] v1.west) -- ([shift={ (0.35,0.25)}] v2.east) -- ([shift={ (0,-0.35)}] v3.south) -- cycle;

                    \filldraw[black, fill=edgecolor, rounded corners,line width=0.1mm] ([shift={ (-0.35,0)}] v8.west) -- ([shift={ (0.15,0.44)}] v9.east) -- ([shift={ (0.15,-0.44)}] v0.east) -- cycle;
                
                \end{scope}
                    
                \begin{scope}[fill opacity=0.7]

                    \node[circle, fill=c1, draw = black, scale=0.7] (v1) at (0,0.1) {};
                    \node[circle, fill=c1, draw = black, scale=0.7] (v2) at (0,-0.0) {};
                    \node[regular polygon, regular polygon sides=3, rotate=180, fill=c6, draw = black, scale=0.6] (v3) at (0,-1) {};
                    \node[star, fill=c3, draw = black, scale=0.6] (v4) at (0,-4) {};
                    \node[regular polygon, regular polygon sides=3, fill=c5, draw = black, scale=0.6pt] (v6) at (0,-5.2) {};
                    \node[rectangle, fill=c4, draw = black, scale=0.8] (v7) at (0,-2.5) {};
                    \node[regular polygon, regular polygon sides=5, fill=c2, draw = black, scale=0.6pt] (v8) at (1.4,-2.5) {};
                    \node[diamond, fill=c7, draw = black, scale=0.6] (v9) at (2.5,-2.5) {};
                    \node[diamond, fill=c7, draw = black, scale=0.6] (v0) at (2.6,-2.5) {};

                \end{scope}

                \draw[black, line width=0.1mm, thin] (v3) -- (v7);
                \draw[black, line width=0.1mm, thin] (v4) -- (v6);               
                \draw[black, line width=0.1mm, thin] (v7) -- (v4);              
                \draw[black, line width=0.1mm, thin] (v7) -- (v8);          
               
                \draw[black, line width=0.1mm, thin] (v6) to[out=-45, in=-135, loop, looseness=8, min distance=7mm] (v6);
                
                \draw[black, line width=0.1mm, thin] (v09) to[out=-45, in=-135, looseness=8,  min distance=7mm] (v00);

                \path[->, >={Straight Barb[length=3pt,width=3pt, flex']}, color = black,line width=0.1mm, thin, bend left=80] (v6) edge node[below] {} (v4);


    \end{tikzpicture} };

    \node[inner sep=2pt, left = 1cm of hpg] (ghostleft) {};
    \node[inner sep=2pt, right = 1cm of quot] (ghostright) {};

\end{tikzpicture}}  
    \caption{\textit{Hypergraph quotient via fibration symmetry.} The left hypergraph shows the original structure where nodes with identical shapes and colors represent distinct vertices belonging to the same fibre under the fibration map $\varphi$ (i.e., nodes receiving equivalent input information). The right side depicts the quotient hypergraph, where nodes with matching shapes and colors represent a single supernode corresponding to an entire fibre. In the quotient representation, slightly overlapping identical nodes within hyperedges indicate multiple instances of the same supernode participating in that hyperedge, preserving the structural information while reducing redundancy through the fibration symmetry $\varphi$. Edges without arrows are undirected (equivalent to having arrows in both directions).}\label{fig: hypergraph_quotient}
\end{figure} 

\begin{oss}\label{obs: comparison_hpg_multi}
 Observe that different types of node partitioning can emerge depending on whether we analyse a hypergraph, its binary graph projection, or its multigraph projection. Each representation captures different structural features of the original hypergraph. The binary graph projection simplifies the hypergraph by representing only the existence of pairwise interactions, discarding information about the number or order of interactions. As a result, it overlooks distinctions between higher-order relationships and repeated interactions, leading to a fibre structure that lacks key information. In contrast, the multigraph projection preserves the multiplicity of interactions between pairs of nodes, capturing how frequently node pairs co-occur in hyperedges. However, it still treats all interactions as equivalent, regardless of their order, thus failing to distinguish the structural role of higher-order interactions.
\end{oss}

\begin{figure}[!ht]
    \centering  
    \scalebox{0.8}{
   \begin{tikzpicture}[scale=0.8, >=latex]
        \node[inner sep=2pt] (hpg) at (0,0)
            {\begin{tikzpicture}[scale=0.8]
                \tikzstyle{point}=[circle,draw=black,fill=black,inner sep=0pt,minimum width=4pt,minimum height=4pt]        
                \node (v1) at (-0.7,0.2) {};
                \node (v2) at (0.7,0.2) {};
                \node (v3) at (0,-1) {};
                \node (v4) at (0,-4) {};
                \node (v5) at (-0.7,-5.2) {};
                \node (v6) at (0.7,-5.2) {};
                \node (v7) at (0,-2.5) {};
                \node (v8) at (1.4,-2.5) {};
                \node (v9) at (2.6,-1.8) {};
                \node (v0) at (2.6,-3.2) {};

                \begin{scope}[fill opacity=0.45]

                    \filldraw[black, fill=edgecolor, rounded corners] ([shift={ (-0.35,0.25)}] v1.west) -- ([shift={ (0.35,0.25)}] v2.east) -- ([shift={ (0,-0.35)}] v3.south) -- cycle;

                    \filldraw[black, fill=edgecolor, rounded corners] ([shift={ (-0.35,0)}] v8.west) -- ([shift={ (0.15,0.44)}] v9.east) -- ([shift={ (0.15,-0.44)}] v0.east) -- cycle;
                
                \end{scope}
                    
                \begin{scope}[fill opacity=0.7]

                    \node[circle, fill=c1, draw = black, minimum size=7pt, label={[label distance=5pt] above left:1}] (v1) at (-0.7,0.2) {};
                    \node[circle, fill=c1, draw = black, minimum size=7pt, label={[label distance=5pt]above right:2}] (v2) at (0.7,0.2) {};
                    \node[regular polygon, regular polygon sides=3, rotate=180, fill=c6, draw = black, minimum size=7pt, label={[label distance=5pt]above left:0}] (v3) at (0,-1) {};
                    \node[star, fill=c3, draw = black, minimum size=9pt, label={[label distance=4pt]left:3}] (v4) at (0,-4) {};
                    \node[regular polygon, regular polygon sides=3, fill=c5, draw = black, minimum size=6pt, label={[label distance=4pt]below left:4}] (v5) at (-0.7,-5.2) {};
                    \node[regular polygon, regular polygon sides=3, fill=c5, draw = black, minimum size=6pt, label={[label distance=4pt]below right:5}] (v6) at (0.7,-5.2) {};
                    \node[rectangle, fill=c4, draw = black, minimum size=7pt, label={[label distance=4pt]left:6}] (v7) at (0,-2.5) {};
                    \node[regular polygon, regular polygon sides=5, fill=c2, draw = black, minimum size=8pt, label={[label distance=4pt]below left:7}] (v8) at (1.4,-2.5) {};
                    \node[diamond, fill=c7, draw = black, minimum size=7pt, label={[label distance=5pt]above right:8}] (v9) at (2.6,-1.8) {};
                    \node[diamond, fill=c7, draw = black, minimum size=7pt, label={[label distance=5pt]below right:9}] (v0) at (2.6,-3.2) {};

                \end{scope}

                \draw[black, line width=0.2mm, thin] (v3) -- (v7);
                \draw[black, line width=0.2mm, thin] (v4) -- (v5);                      
                \draw[black, line width=0.2mm, thin] (v4) -- (v6);               
                \draw[black, line width=0.2mm, thin] (v5) -- (v6);              
                \draw[black, line width=0.2mm, thin] (v7) -- (v4);              
                \draw[black, line width=0.2mm, thin] (v9) -- (v0);            
                \draw[black, line width=0.2mm, thin] (v7) -- (v8);
                
            \node[label={[text width=4cm, align=center, text centered]below:{Hypergraph Fibres}}] (1) at (0,-6) {};   
                
            \end{tikzpicture}};

         \node[inner sep=0pt] (graph) at (7,0)
            {\begin{tikzpicture}[scale=0.8]
                 \tikzstyle{point}=[circle,draw=black,fill=black,inner sep=0pt,minimum width=4pt,minimum height=4pt]   

                 \begin{scope}[fill opacity=0.7]

                    \node[circle, fill=c1, draw = black, minimum size=7pt, label={[label distance=5pt] above left:1}] (v1) at (-0.7,0.2) {};
                    \node[circle, fill=c1, draw = black, minimum size=7pt, label={[label distance=5pt]above right:2}] (v2) at (0.7,0.2) {};
                    \node[regular polygon, regular polygon sides=3, rotate=180, fill=c6, draw = black, minimum size=11pt, label={[label distance=5pt]above left:0}] (v3) at (0,-1) {};
                    \node[regular polygon, regular polygon sides=3, rotate=180, fill=c6, draw = black, minimum size=11pt, label={[label distance=4pt]left:3}] (v4) at (0,-4) {};
                    \node[circle, fill=c1, draw = black, minimum size=7pt, label={[label distance=4pt]below left:4}] (v5) at (-0.7,-5.2) {};
                    \node[circle, fill=c1, draw = black, minimum size=7pt, label={[label distance=4pt]below right:5}] (v6) at (0.7,-5.2) {};
                    \node[rectangle, fill=c4, draw = black, minimum size=7pt, label={[label distance=4pt]left:6}] (v7) at (0,-2.5) {};
                    \node[regular polygon, regular polygon sides=3, rotate=180, fill=c6, draw = black, minimum size=11pt, label={[label distance=5pt]above right:7}] (v8) at (1.4,-2.5) {};
                    \node[circle, fill=c1, draw = black, minimum size=7pt, label={[label distance=5pt]above right:8}] (v9) at (2.6,-1.8) {};
                    \node[circle, fill=c1, draw = black, minimum size=7pt, label={[label distance=5pt]below right:9}] (v0) at (2.6,-3.2) {};

                \end{scope}

                \draw[black, line width=0.2mm, thin] (v3) -- (v7);
                \draw[black, line width=0.2mm, thin] (v4) -- (v5);                      
                \draw[black, line width=0.2mm, thin] (v4) -- (v6);               
                \draw[black, line width=0.2mm, thin] (v5) -- (v6);              
                \draw[black, line width=0.2mm, thin] (v7) -- (v4);              
                \draw[black, line width=0.2mm, thin] (v9) -- (v0);            
                \draw[black, line width=0.2mm, thin] (v7) -- (v8);            
                \draw[black, line width=0.2mm, thin] (v1) -- (v2);              
                \draw[black, line width=0.2mm, thin] (v2) -- (v3);            
                \draw[black, line width=0.2mm, thin] (v1) -- (v3);            
                \draw[black, line width=0.2mm, thin] (v8) -- (v9);            
                \draw[black, line width=0.2mm, thin] (v8) -- (v0);

               \node[label={[text width=3cm, align=center, text centered]below:{Graph Fibres}}] (2) at  (0,-6) {};   
                    
                \end{tikzpicture}};       

         \node[inner sep=0pt] (multigraph) at (14,0)
            {\begin{tikzpicture}[scale=0.8]
                \tikzstyle{point}=[circle,draw=black,fill=black,inner sep=0pt,minimum width=4pt,minimum height=4pt]   

                \begin{scope}[fill opacity=0.7]

                    \node[circle, fill=c1, draw = black, minimum size=7pt, label={[label distance=5pt] above left:1}] (v1) at (-0.7,0.2) {};
                    \node[circle, fill=c1, draw = black, minimum size=7pt, label={[label distance=5pt]above right:2}] (v2) at (0.7,0.2) {};
                    \node[regular polygon, regular polygon sides=3, rotate=180, fill=c6, draw = black, minimum size=11pt, label={[label distance=5pt]above left:0}] (v3) at (0,-1) {};
                    \node[regular polygon, regular polygon sides=3, rotate=180, fill=c6, draw = black, minimum size=11pt, label={[label distance=4pt]right:3}] (v4) at (0,-4) {};
                    \node[circle, fill=c1, draw = black, minimum size=7pt, label={[label distance=4pt]below left:4}] (v5) at (-0.7,-5.2) {};
                    \node[circle, fill=c1, draw = black, minimum size=7pt, label={[label distance=4pt]below right:5}] (v6) at (0.7,-5.2) {};
                    \node[rectangle, fill=c4, draw = black, minimum size=7pt, label={[label distance=4pt]left:6}] (v7) at (0,-2.5) {};
                    \node[regular polygon, regular polygon sides=5, fill=c2, draw = black, minimum size=9pt, label={[label distance=4pt]below left:7}] (v8) at (1.4,-2.5) {};
                    \node[diamond, fill=c7, draw = black, minimum size=8pt, label={[label distance=5pt]above right:8}] (v9) at (2.6,-1.8) {};
                    \node[diamond, fill=c7, draw = black, minimum size=8pt, label={[label distance=5pt]below right:9}] (v0) at (2.6,-3.2) {};

                \end{scope}

                \draw[black, line width=0.2mm, thin] (v3) -- (v7);
                \draw[black, line width=0.2mm, thin] (v4) -- (v5);                      
                \draw[black, line width=0.2mm, thin] (v4) -- (v6);               
                \draw[black, line width=0.2mm, thin] (v5) -- (v6);              
                \draw[black, line width=0.2mm, thin] (v7) -- (v4);             
                \draw[black, line width=0.2mm, thin] (v7) -- (v8);            
                \draw[black, line width=0.2mm, thin] (v1) -- (v2);              
                \draw[black, line width=0.2mm, thin] (v2) -- (v3);            
                \draw[black, line width=0.2mm, thin] (v1) -- (v3);            
                \draw[black, line width=0.2mm, thin] (v8) -- (v9);            
                \draw[black, line width=0.2mm, thin] (v8) -- (v0);

                \path[color = black, bend left=15, line width=0.2mm, thin] (v9) edge node[above] {} (v0);
                \path[color = black, bend left=15, line width=0.2mm, thin] (v0) edge node[above] {} (v9);

        \node[label={[text width=4cm, align=center, text centered]below:{Multigraph Fibres}}] (2) at  (0,-6) {};  
                    
                \end{tikzpicture}};

        \node[inner sep=2pt, left = 1cm of hpg] (ghostleft) {};
        \node[inner sep=2pt, right = 1cm of graph] (ghostright) {};

    \end{tikzpicture}
}  
    \caption{\textit{Different fibre partitions according to the chosen representation.} Example of fibres for a hypergraph, its projection to a simple graph, and its projections to a graph with multiple edges. In each representation, nodes belonging to the same fibres are pictured with the same color and icon.}\label{fig: fibers_differences}
\end{figure} 

 This behavior is illustrated with an example in \cref{fig: fibers_differences}. The left panel shows a hypergraph, where nodes are grouped into fibres based on their participation in higher-order hyperedges (represented by shaded regions). In the middle panel, the binary graph projection loses both the frequency and the order of these interactions, resulting in a different, less informative fibre structure. The right panel, representing the multigraph projection, captures repeated pairwise connections but still ignores the original hyperedge order, leading to distinct partitioning.

\subsection{Algorithm}
In \cref{sec: cat_eq}, we discussed how, thanks to the equivalence of a hypergraph and its bipartite projection, we can utilize established algorithms, such as those presented in~\cite{kamei_computation_2013}.
In this section, we report an algorithm to obtain the partition of the nodes of an input hypergraph according to the fibres.
The following \cref{algo: fibers_assign} proceeds by projecting the hypergraph into its bipartite representation and then applying an iterative color refinement scheme. The notation and the steps follow the algorithm presented in~\cite{morone_fibration_2020}. Hence, the complexity to assign colors to nodes and hyperedges of a hypergraph $H = (N, E)$ scales as $O(\lvert F \rvert \, log (\lvert N \rvert + \lvert E \rvert))$, where $F$ is the set of edges in the bipartite representation of $H$.

\begin{oss}
    The algorithm of Morone et al.~\cite{morone_fibration_2020} may fail only in the rare presence of multiple disconnected components or distinct nodes without inputs, as shown in counterexample~7.3 of~\cite{leifer_symmetryinspired_2022}.
    In our case, this failure does not occur, since we only consider the undirect case with a single connected component. Moreover, as implemented in our algorithm (see \cref{sec: code}), the initial color dictionary is provided as input to distinguish between original nodes and hyperedges, so assigning distinct initial colors to source nodes would, in any case, prevent such ambiguity.\\
\end{oss}

\begin{algorithm}[H]
\SetAlgoLined
\KwIn{Hypergraph $H = (N, E)$, where $N$ are the nodes and $E$ are the hyperedges.}
\KwOut{$\{c_i\}_{i \in N}$, where $c_i$ is the color of node $i$ in $N$.
$\{c_j\}_{j \in E}$, where $c_j$ is the color of hyperedge $j$ in $E$.}

\LinesNumberedHidden
\textbf{Step 1: Transform the hypergraph into its bipartite graph projection.} \\
$G = (N \sqcup E, F)$ be the bipartite graph, where:
\begin{itemize}[noitemsep]
    \item Nodes of $G$ consist of two disjoint sets: $N$ (nodes of the hypergraph) and $E$ (hyperedges of the hypergraph).
    \item Edges $F$ connect nodes in $N$ to the hyperedges in $E$ they belong to.
\end{itemize}
\textbf{Step 2: Assign starting colors.} \\
\ForEach{node $v \in N \sqcup E$}{
    \uIf{$v \in N$ (nodes of the first layer)}{
        Assign initial color $c_v = 0$.
    }
    \ElseIf{$v \in E$ (nodes of the second layer)}{
        Assign initial color $c_v = 1$.
    }
}

\textbf{Step 3: Apply the iterative coloring algorithm.} \\
$j \gets 0$ \\
$R_0 \gets 1$ \\

\Repeat{$R_j \neq R_{j-1}$}{
    \For{$i = 1$ \KwTo $|N \sqcup E|$, $k = 0, \dots, R_j$}{
        $I_i^k \gets$ number of nodes of color $k$ in the input set of $i$
    }

    $I \gets$ set of all unique $I_i \defeq \{I_i^k\}_k$ (an ordered list of $I_i^k$ ) \\

    \For{$i = 1$ \KwTo $|N|$}{
        $c_i \gets$ index of $I_i$ in $I$ (e.g., if two nodes have the same $I_i$ and $I_j \to c_i = c_j$)
    }

    $j \gets j + 1$ \\
    $R_j \gets |I|$
}

\textbf{Step 4: Extract final colors for nodes of the first layer and for nodes of the second layer.} \\
\Return{$\{c_i \mid i \in N\}$ (colors of the hypergraph nodes), $\{c_j \mid j \in E\}$ (colors of the hyperedges).}

\caption{Finding fibres for a hypergraph via its bipartite graph projection
.}\label{algo: fibers_assign}
\end{algorithm}

\subsection{Fibration symmetries for real-world datasets with HOI}\label{sub: real}
To illustrate the relevance of fibration symmetries beyond theoretical models, we analyse several real-world systems featuring higher-order interactions.
We computed the nodes partition via fibration symmetry for some datasets with higher-order interactions provided by
\href{https://github.com/xgi-org/xgi-data}{XGI dataset}~\cite{landry_xgi_2023}. In addition, we calculated the fibres for an emblematic example of higher-order interactions in real-world situations: collaborations among researchers for publishing a paper in a conference. We considered the MAG-10 dataset~\cite{amburgilya_clustering_2020, sinha_overview_2015}, preprocessed in \url{https://github.com/TheoryInPractice/overlapping-ecc/tree/master/data/MAG-10}, a subset of the Microsoft Academic Graph in which the nodes are the authors (labelled with a code), the hyperedges correspond to a publication by those authors, and the hyperedges are categorised according to one of the 10 computer science conferences, the most
common conference in which they published. In particular, we considered the \textit{International Conference on Machine Learning} (ICML), filtering by code 3, the \textit{International World Wide Web Conference} (WWW), filtering by code 2, and the \textit{International Conference on Computer Vision} (ICCV) with code 8. In this dataset, articles with more than 25 authors are omitted.

\begin{table}[ht]\label{tab: real_dataset}
    \centering
    \begin{tabular}{lllll}
    \toprule
    Dataset & $\lvert N \rvert$ &  $\lvert E \rvert$  &  $\lvert N \rvert /  \lvert C \rvert$ & \# Non-trivial fibres \\
    \midrule    
    diseasome & 516 & 314 & 1.46 & 94 \\
    disgenenet & 12368 & 1672 & 1.39 & 629 \\    
    hospital-lyon & 75 & 1824 & 1.00 & 0 \\
    house-committees & 1290 & 335 & 1.07 & 51 \\
    hypertext-conference & 113 & 2434 & 1.00 & 0 \\
    invs13 & 92 & 787 & 1.00 & 0 \\
    MAG-10 (ICML) & 9981 & 4803 & 2.32 & 1354 \\
    MAG-10 (WWW) & 12889 & 5468 & 2.42 & 1829 \\
    MAG-10 (ICCV) & 11693 & 5367 & 2.30 & 1683 \\
    malawi-village & 84 & 431 & 1.01 & 1 \\
    ndc-classes & 628 & 796 & 1.51 & 165 \\
    ndc-substances & 3414 & 6471 & 1.16 & 248 \\
    plant-pollinator-mpl-015 & 130 & 401 & 1.02 & 2 \\
    plant-pollinator-mpl-016 & 26 & 73 & 1.04 & 1 \\
    plant-pollinator-mpl-049 & 37 & 91 & 1.00 & 0 \\
    plant-pollinator-mpl-062 & 456 & 866 & 1.00 & 0 \\
    science-gallery & 410 & 3350 & 1.00 & 1\\
    \bottomrule
    \end{tabular}
    \caption{\textit{Summary table showing the nodes partition in fibres across different datasets with HOI.} Table reporting for each dataset, the number of nodes $\lvert N \rvert$, the number of edges $\lvert E \rvert$ excluding singletons. Considering $\lvert C \rvert$  as the number of fibres partitioning the nodes via higher-order symmetry fibrations, we report the average number of nodes in each fibre $\lvert N \rvert /  \lvert C \rvert$ and, in the last column, the number of nontrivial fibres, all the fibres formed by more than one node.}
\end{table}

As observed in \cref{obs: comparison_hpg_multi}, for the MAG-10 dataset related to the ICML conference, representing paper collaborations through a hypergraph rather than a multigraph alters the fibres. 
The case depicted in \cref{fig: coauthorship_comparison} illustrates how, in the multi-edge projection, the two authors with codes 7979 and 4508 are placed in the same fibre, despite having collaborated in different ways with the same authors (with codes 4509, 7980). Indeed, while the author 7979 collaborated on one paper with the other two authors 4509, 7980 together, the author 4508 produced two papers in two different collaborations with each of the individual authors. The graph representation with multiple edges flattens this case by not recognizing the different types of collaboration by bringing them together.
An analogue mismatch occurs for the dataset corresponding to the publication in the International World Wide Web Conference. Specifically, at WWW, the author identified as 21331 co-authored two separate papers with authors 21332 and 22750, while the author with code 32176 collaborated with both on a single paper. In the multigraph representation, a fibre includes authors 32176 and 21331, failing to differentiate their distinct publishing contributions.
The same situation occurs for the two pairs of authors with code 18628, 51267 and 2149, 37684 related to ICCV.

\begin{figure}[ht]
    \centering
    \scalebox{0.8}{   \begin{tikzpicture}[scale=1, >=latex]
   \centering{
        \node[inner sep=2pt] (hpg) at (-3,0)
            {\begin{tikzpicture}[scale=1]
                \tikzstyle{point}=[circle,draw=black,fill=black,inner sep=0pt,minimum width=4pt,minimum height=4pt]   
                \node (v1) at (-0.7,0.2) {};
                \node (v2) at (0.7,0.2) {};
                \node (v3) at (0,-1) {};   

                \begin{scope}[fill opacity=0.45]

                    \filldraw[black, fill=edgecolor, rounded corners] ([shift={ (-0.35,0.25)}] v1.west) -- ([shift={ (0.35,0.25)}] v2.east) -- ([shift={ (0,-0.35)}] v3.south) -- cycle;

                \end{scope}
                    
                \begin{scope}[fill opacity=0.8]

                    \node[diamond, fill=c2, draw = black, minimum size=10pt, label={[label distance=7pt] left:4509}, outer sep=2pt] (v1) at (-0.7,0.2) {};
                    \node[star, fill=c3, draw = black, minimum size=7pt, label={[label distance=7pt]right:7980}, outer sep=2pt] (v2) at (0.7,0.2) {};
                    \node[circle, fill=c1, draw = black, minimum size=7pt, label={[label distance=7pt]below:7979}, outer sep=2pt] (v3) at (0,-1) {};
                    \node[rectangle, fill=c5, draw = black, minimum size=7pt, label={[label distance=7pt]above:4508}, outer sep=2pt] (v4) at (0,1.4) {};

                    \draw[black, line width=0.2mm, thin]  (v4) -- (v1);
                    \draw[black, line width=0.2mm, thin]  (v4) -- (v2);

                    \path (v4) -- (v1) coordinate[pos=0.5] (e1);
                    \path (v4) -- (v2) coordinate[pos=0.5] (e2);
        
                    \node[inner sep=0pt, fill = white, line width=0.1pt] at (e1) {\footnotesize\faFile[regular]};
                    \node[inner sep=0pt, fill = white, line width=0.1pt] at (e2) {\footnotesize\faFile[regular]};
        
                    \node[inner sep=0pt, fill = white, line width=0.1pt] at (0,-0.27) {\footnotesize\faFile[regular]};

                    \node[inner sep=0pt, fill = white, line width=0.1pt]  (ghost) at (1.5,0.8) {};
                    \node[inner sep=0pt, fill = white, line width=0.1pt]  (ghost2) at (1.5,-0.4) {};

                    \draw[line width=0.1mm, dashed] (v2) -- (ghost) ;
                    \draw[line width=0.1mm, dashed] (v2) -- (ghost2);

                \end{scope}

            \node[label={[text width=4cm, align=center, text centered]below:{Hypergraph\\[-1pt]Coauthors Fibres}}] (1) at (0,-2.5) {};   
                
            \end{tikzpicture}};

         \node[inner sep=0pt] (graph) at (3,0)
            {\begin{tikzpicture}[scale=1]
                \tikzstyle{point}=[circle,draw=black,fill=black,inner sep=0pt,minimum width=4pt,minimum height=4pt]   
                    
                \begin{scope}[fill opacity=0.8]            

                    \node[diamond, fill=c2, draw = black, scale=10, label={[label distance=7pt] left:4509}] (v1) at (-0.7,0.2) {};
                    \node[star, fill=c3, draw = black, scale=9, label={[label distance=7pt]right:7980}] (v2) at (0.7,0.2) {};
                    \node[circle, fill=c1, draw = black,scale=7, label={[label distance=7pt]below:7979}] (v3) at (0,-1) {};
                    \node[circle, fill=c1, draw = black,scale=7, label={[label distance=7pt]above:4508}] (v4) at (0,1.4) {};

                    \draw[black, line width=0.2mm, thin] (v4) -- (v1);
                    \draw[black, line width=0.2mm, thin] (v4) -- (v2);     
                    \draw[black, line width=0.2mm, thin] (v1) -- (v3);
                    \draw[black, line width=0.2mm, thin] (v2) -- (v3);    
                    \draw[black, line width=0.2mm, thin] (v1) -- (v2);   

                    \path (v4) -- (v1) coordinate[pos=0.5] (e1);
                    \path (v4) -- (v2) coordinate[pos=0.5] (e2);
                    \path (v1) -- (v3) coordinate[pos=0.5] (e3);
                    \path (v2) -- (v3) coordinate[pos=0.5] (e4);
                    \path (v1) -- (v2) coordinate[pos=0.5] (e5);
        
                    \foreach \e in {e1,e2,e3,e4,e5} {
                        \node[inner sep=0pt, thin, fill = white, line width=0.1pt] at (\e) {\footnotesize\faFile[regular]};
                    }

                    \node[inner sep=0pt, fill = white, line width=0.1pt]  (ghost) at (1.5,0.8) {};
                    \node[inner sep=0pt, fill = white, line width=0.1pt]  (ghost2) at (1.5,-0.4) {};

                    \draw[line width=0.1mm, dashed] (v2) -- (ghost) ;
                    \draw[line width=0.1mm, dashed] (v2) -- (ghost2);

                \end{scope}

                
               \node[label={[text width=4cm, align=center, text centered]below:{Multigraph\\[-1pt]Coauthors Fibres}}] (2) at  (0,-2.5) {};  
                    
                \end{tikzpicture}};

        \node[inner sep=2pt, left = 3cm of hpg] (ghostleft) {};
        \node[inner sep=2pt, right = 3cm of graph] (ghostright) {};
  
    }
\end{tikzpicture}} 
    \caption{\textit{Visualization of different fibre partitions based on the representation of the co-author relationship.} Example of different fibres between the hypergraph representation and its multigraph projection in a real-world case for collaboration on publications for the ICML conference (MAG-10 dataset filtered for ICML conference). Nodes represent authors, and paper collaborations are depicted with a hyperedge among the authors, stressed with a paper icon. Nodes painted with the same color and mark belong to the same fibre. For the sake of the example, in this image, we represent only the papers related to authors labeled with codes 7979 and 4508. Nodes 4509 and 7980 are in different fibres in both cases as they are related (illustrated with dashed lines) to different other nodes not present in the picture.}\label{fig: coauthorship_comparison}
\end{figure}

\section{Cluster synchronization}
Having established the structural foundations through definitions and illustrative examples of fibrations for higher-order interactions, we now turn to the dynamical aspects of the system. 
To ensure compatibility between the dynamics and the underlying hypergraph structure, we consider admissible ODE systems~\cite{aguiar_network_2023, vondergracht_hypernetworks_2023}, where the evolution of each node depends only on the structure-prescribed interactions. The higher-order Kuramoto model with frustration (Kuramoto-Sakaguchi) falls within this class, as its dynamics respect the hypergraph topology and involve only interactions explicitly defined by the system’s combinatorial structure.
In this section, we introduce the dynamical model and examine how the underlying structural features influence its behavior. We aim to connect the structural characterization of robust synchrony due to Von der Gracht et al.\ and Aguiar et al. with the behavior of the higher-order Kuramoto dynamics, showing in particular that synchrony in this model with homogeneous initial conditions and natural frequencies is constrained to the fibration partition.

\subsection{Higher-order Kuramoto model with phase-lag parameters}
We study synchronization dynamics in oscillator populations using a Kuramoto model with both pairwise and three-body interactions. Although our framework refers broadly to higher-order interactions, we focus on two- and three-body terms as the minimal setting that captures collective behaviors beyond standard pairwise coupling. Our primary interest lies in cluster and remote synchronization, where phase alignment of identical oscillators occurs only within specific groups of nodes, rather than across the entire network. The inclusion of three-body interactions follows the principled derivation of León et al.~\cite{leon_higherorder_2024}, who used phase reduction techniques to show that such interactions naturally arise in the form
$\sin\big(\theta_{\gamma} + \theta_{\delta}- 2\theta_{a} -\alpha \big)$ for an oscillator $a$. While León et al.’s formulation assumes all-to-all coupling, we generalize it to arbitrary network topologies, in line with Muolo et al.~\cite{muolo_when_2025}. Crucially, we retain the phase-lag parameter, which prevents trivial global synchronization and enables the network to exhibit nontrivial collective states, even with identical oscillators starting with identical initial conditions.\\

\paragraph*{Model}
Consider a hypergraph $H = (N, E)$, with rank $rk(H)\leq 3$. Define the incidence matrix of order $m$ with $m \in \{2,3\}$ as the matrix $I^{(m)}$ with as many rows as the cardinality of $N$ and as many columns as the cardinality of $\mathcal{E}^{(m)} = \big\{ e \in E \mid \lvert e \rvert = m \big\}$, that is the set of hyperedges of order $m$.
    Then $\forall i \in N, e \in E$:
    \begin{equation*}
        [I^{(m)}]_{i,e} = 
            \begin{cases}
    			1 & \text{if}\,\, i \in e\\
                0 & \text{otherwise}
    		\end{cases}
    \end{equation*}

The evolution of each node $i \in N$ can be expressed as follows:
    \begin{align}\label{eq: dynamics_frustration}
    	\dot{\theta}_i = \omega_i + \sum\limits_{m=2}^{3}\sigma_m\sum\limits_{e\in \mathcal{E}^{(m)}} [I^{(m)}]_{i,e} \sin{\Big(\sum_{j \in e\setminus \{i\}}\theta_j - (m-1) \theta_{i} -\alpha^{(m)} \Big)},
    \end{align}        
where $\sigma_m$ is the strength of the $m$th order coupling and $\omega_i$ is the natural frequency of the oscillator associated with the node. Angles $\alpha^{(m)} \in (0,\,\pi/2)$ are the phase frustration parameters for each order $m$. The frustration parameters prevent the system from achieving global synchronization (unless all nodes belong to a single trivial cluster), instead allowing the emergence of partial or cluster synchronization patterns. To isolate the role of the network structure and frustration in shaping these patterns, we will assign identical initial conditions and natural frequencies to all nodes. This choice reflects an anonymity condition in the sense of~\cite{makse_symmetries_2025}, ensuring that any observed symmetry breaking arises purely from the dynamics and topology, rather than from imposed heterogeneity. 

\begin{oss}
We observe that the coupling function of the Kuramoto model for identical oscillators with phase-lags $\alpha^{(m)} = 0$ for all $m$ is non-invasive (i.e., evaluated in $\theta_i = \theta^*$ for all $i$ is identically 0). Hence, a solution exists for global system synchronization. 
By inserting non-zero $\alpha^{(m)}$ instead, we avoid this case.
\end{oss}

\paragraph*{Synchronization measures}
A common way to measure coherence among a subset of nodes evolving according to a Kuramoto model is to compute the Kuramoto order parameters.
We will compare global and cluster order parameters to analyse their evolution in time.
\begin{defn}[Kuramoto Order parameters]
The global Kuramoto order parameter for oscillators associated to nodes in a hypergraph $H = (N, E)$ is defined as
\begin{equation*}
R(t) = \left| \frac{1}{\lvert N \rvert} \sum_{j\in N} e^{i\theta_j(t)} \right|
\end{equation*}
where $R(t) \in [0, 1]$ quantifies the overall phase coherence of the system.
Given a cluster $ C_\alpha \subseteq N $, the local order parameter is
\begin{equation*}
R_\alpha(t) = \left| \frac{1}{|C_\alpha|} \sum_{j \in C_\alpha} e^{i\theta_j(t)} \right|
\end{equation*}
where $ R_\alpha(t) \in [0, 1] $ measures the synchronization within cluster $ C_\alpha $.
\end{defn}

The Kuramoto order parameter is a measure of phase coherence among a group of oscillators. It takes values between 0 and 1, where 1 indicates perfect synchronization, all nodes are in phase, while values closer to 0 reflect increasing incoherence.

\cref{fig: frustration_synt,fig: frustration_real_diseasome} present the temporal evolution of global and local order parameters for fibre configurations across two hypergraph architectures: synthetic networks, and the diseasome.
In all the images, we can see that at $t=0$ all the parameters are 1, as the initial conditions are set equal for all nodes. During the evolution, local order parameters for fibre clusters remain constant at unity across all hypergraph types, while global order parameters exhibit deviations from unity. This divergence indicates that global coherence breaks down despite the preservation of local cluster organization, indicating the influence of network topology on these dynamics.

\begin{figure}[!ht]
    \centering
    \includegraphics[width=0.8\linewidth]{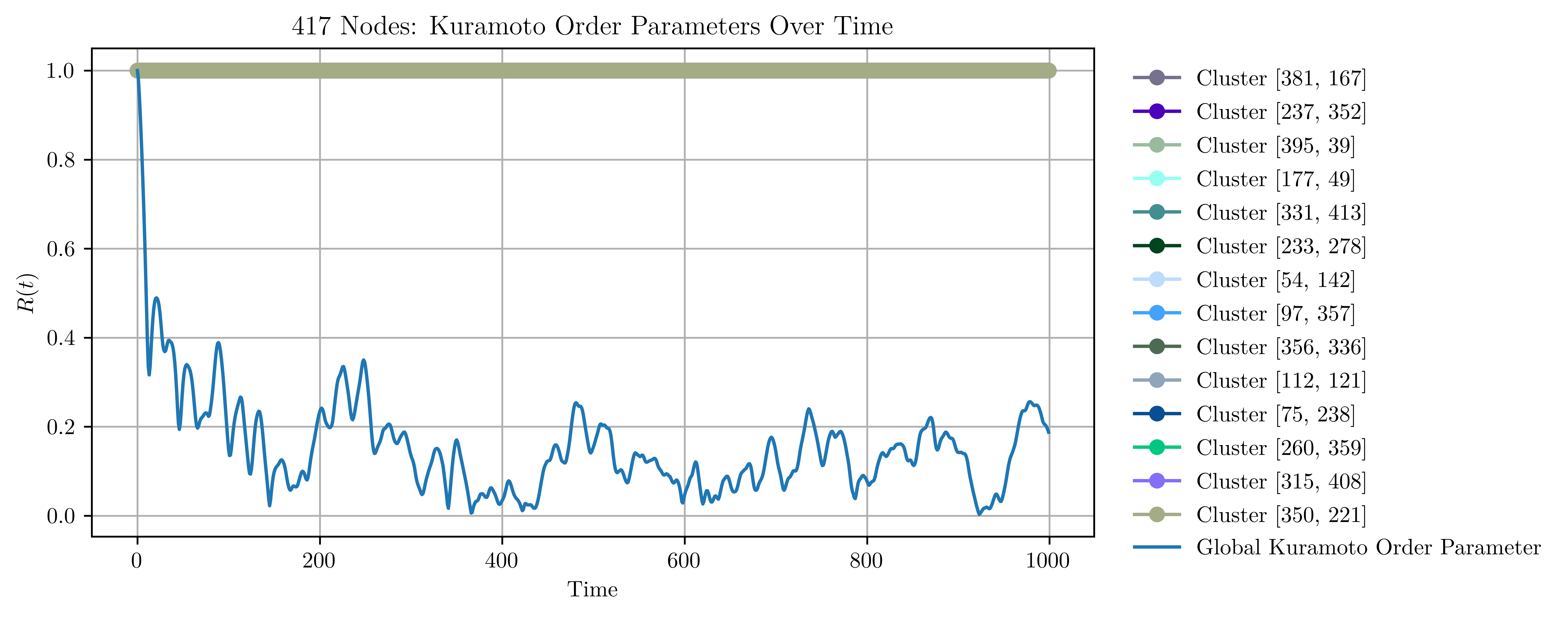}
    \caption{\textit{Cohesiveness measure in fibres and globally for a synthetic hypergraph.} The Kuramoto order parameters $R$ for different fibres and for the whole hypergraph for a higher-order Kuramoto model with frustration (\cref{eq: dynamics_frustration} with $\lvert N \rvert = 417$, $\alpha = \pi/6$ and $(\sigma_2, \sigma_3) = (0.2,0.6)$). When $R$ is closer to 1, the phases of all nodes are more synchronized. The local order parameters for all fibres overlap and are identically 1. The blue thin line is the global order parameter.}\label{fig: frustration_synt}
\end{figure}

\begin{figure}[!ht]
    \centering
    \includegraphics[width=0.6\linewidth]{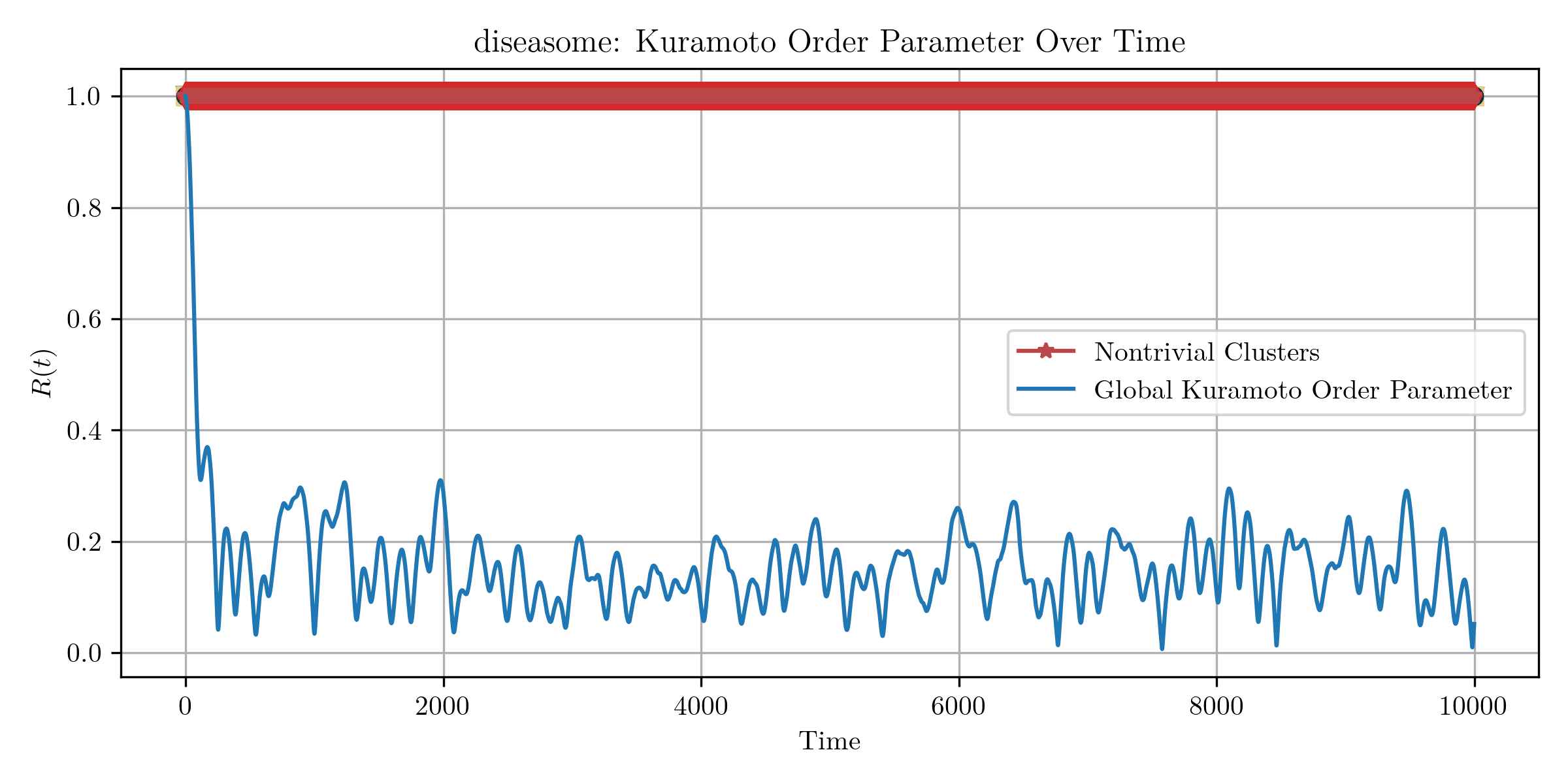}
    \caption{\textit{Cohesiveness measure in fibres and globally for the diseasome dataset.} The Kuramoto order parameters $R$ for the nontrivial fibre and for the whole hypergraph with a higher-order Kuramoto model with frustration (\cref{eq: dynamics_frustration_fibers}, with $\alpha = \pi/6$). The topology is from the largest connected component without duplicate edges of the \textit{diseasome} dataset (94 nontrivial clusters)~\cite{goh_human_2007}. In this dataset, a disease is a node, and a gene is a hyperedge. When $R$ is closer to 1, the phases of all nodes are more synchronized. In red, the overlapped fibre order parameters are identically equal to 1 for all 94 non-trivial fibres.}\label{fig: frustration_real_diseasome}
\end{figure}

\subsection{Fibration and Kuramoto synchrony partitions in the Kuramoto model}
Let's analyse the relations between the nodes' partition of a hypergraph in fibres and a Kuramoto synchrony partition starting from identical initial conditions and natural frequencies for all nodes. We first recall that the implication \textit{fibration partition $\Rightarrow$ Kuramoto synchrony partition} (\cref{thm: fibres_synch}) follows directly from the general theory~\cite{aguiar_network_2023} as the vector field associated with our model is admissible. We include a proof in the specific Kuramoto setting for completeness. The second implication proved in \cref{thm: synch_fibres_identicalIC} addresses the converse direction: in our restricted Kuramoto framework with homogeneous initial conditions and natural frequencies, any Kuramoto synchrony partition must coincide with the fibration partition, except on a parameter set of measure zero. To introduce the notion of a subspace of the hypernetwork’s total phase space characterized by the equality of node phases, we define a Kuramoto synchrony partition of the nodes according to their trajectories. We generalize the definition of $\alpha$-Kuramoto partition~\cite{kirkland_alphakuramoto_2013} to emphasize the dynamical nature of these partitions in the context of the Kuramoto model with frustration.

\begin{defn}[Kuramoto synchrony partition]
    Given an undirected hypergraph $H=(N,E)$ with rank $rk(H)\leq 3$, take a set of $2$ parameters $\{\alpha^{(m)}\}_{m \in \{2,3\}}$ with every $\alpha^{(m)} \in (0,\pi/2)$. Consider the trajectories $\theta_i$ for the nodes resulting from the \cref{eq: dynamics_frustration} with parameters $\{\alpha^{(m)}\}_{m}$, $(\sigma_2, \sigma_3) \neq (0,0)$ identical natural frequencies $\omega_i = \omega$ and initial conditions $\theta_i(0) = \theta^0_i, \quad i \in N$ for all nodes. We call Kuramoto synchrony partition the node partition $N= S_1 \sqcup \dots \sqcup S_h$ such that:
    \begin{itemize}
        \item $i, j \in S_a$ if $\theta_i(t) = \theta_j(t) \,\, \forall t \geq 0$.
        \item if $i\in S_a$ and $\ell \in S_b$ with $a \neq b$, then there exists some $t \geq 0$ such that $\theta_i(t) \neq \theta_\ell(t)$.        
    \end{itemize}
\end{defn}

In a nutshell, nodes with the same trajectories for all the process are in the same cluster, while nodes with different trajectories must be in different clusters. Observe that belonging to the same cluster results in the same trajectory for all $t$, which we can indicate with the cluster index $\theta_i(t) = \theta_a (t)$ for all nodes $i$ in cluster $S_a$. We can reduce the number of frustration parameters to just one $\alpha$ angle belonging to $(0,\pi/2)$, as non-zero frustration is sufficient to prevent global synchronization and avoid non-invasiveness of the coupling function.
As we can see in the example in \cref{fig: frustration_heatmap}, having non-null frustration parameters prevents global synchronization throughout time.
Hence, for simplicity, for all $m$, we can set $\alpha^{(m)} = \alpha$.

\begin{figure}[!ht]
    \centering
    \includegraphics[width=0.56\linewidth]{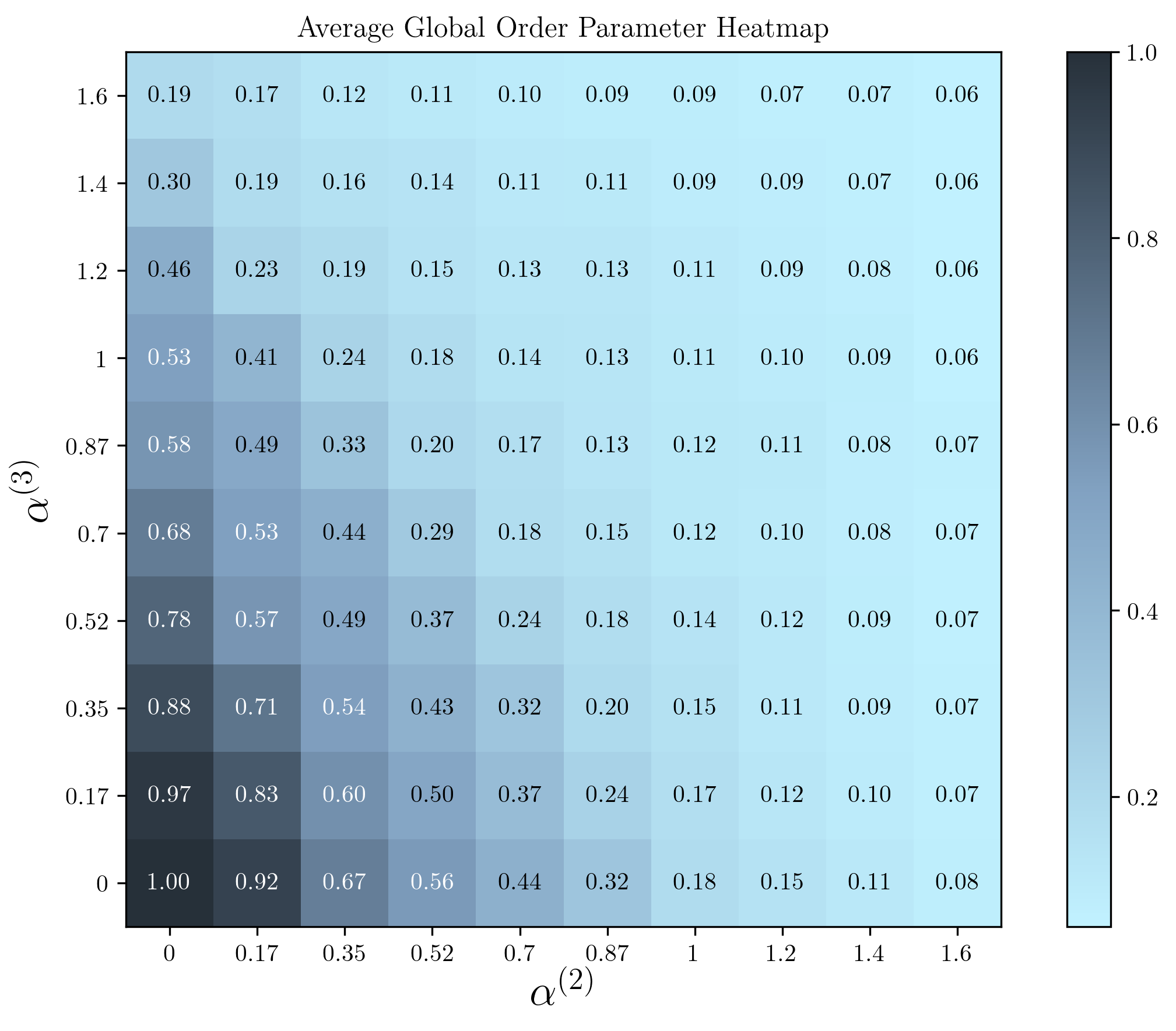}
    \caption{\textit{Global order parameters through varying coupling force parameter options.} The average global order parameter $\langle R(t) \rangle_t$  computed over 500 time steps when changing the frustration parameters of \cref{eq: dynamics_frustration} with $\alpha_2, \alpha_3 \in [0, \pi/2)$. When $\langle R(t) \rangle_t$ approaches 1, it indicates near-complete phase alignment across all nodes.}\label{fig: frustration_heatmap}
\end{figure}

We can also restate the evolution for a node $i$ in the cluster $S_a$ of the Kuramoto synchrony partition (\cref{eq: dynamics_frustration}), emphasizing the node's membership in a cluster.
\begin{align}\label{eq: dynamics_frustration_fibers}
    \dot{\theta}_i = \omega_i + \sigma_2\sum\limits_{\beta=1}^h k^{(2)}_{i\,\beta} \sin\big({\theta_{\beta} - \theta_{a} -\alpha \big)} + \sigma_3\sum\limits_{\gamma = 1} ^h\sum\limits_{\delta = \gamma }^h k^{(3)}_{i\,\gamma \,\delta} \sin\big(\theta_{\gamma} + \theta_{\delta}- 2\theta_{a} -\alpha \big).
\end{align}    

where the double sum with $\delta \ge \gamma$ avoids double-counting symmetric pairs of clusters in three-body interactions, after fixing an arbitrary ordering of the clusters $\{S_1, \dots, S_h\}$. We denote $k^{(m)}_{i\,\gamma_1 \cdots \,\gamma_{m-1}}$ as the number of nodes sharing a hyperedge of order $m$ with node $i$ and belonging to the clusters specified by the remaining $m-1$ indices of $k$.
It can be considered as the degree of the node $i$ of order $m$ linking $i$ to the different subsets of clusters.
We can prove the first implication by generalizing the proof of~\cite{kirkland_alphakuramoto_2013}. This theorem can also be simply derived from the general theory of~\cite{aguiar_network_2023}.

\begin{thm}[Fibration partition $\Rightarrow$ Kuramoto synchrony partition]\label{thm: fibres_synch} 
Let $H=(N,E)$ be a connected hypergraph with rank at most 3. Consider the fibration-partition $\{C_1, \dots, C_h\}$ of the nodes given by the fibres of the hypergraph. Then for any $\alpha \in (0,\pi/2)$, $(\sigma_2, \sigma_3) \neq (0,0)$ the fibration-partition is a Kuramoto synchrony partition.
\end{thm}

\begin{proof}
    Consider a system of equations associated with the fibres $C_a$ for $a \in \{1, \dots , h\}$ with initial conditions $x_a(0) = x_0$ and natural frequency $\omega$:

\begin{align*}
\dot{x}_a = \omega + \sigma_2\sum\limits_{\beta=1}^h k^{(2)}_{a\,\beta} \sin\big({x_{\beta} - x_{a} -\alpha \big)} + \sigma_3\sum\limits_{\gamma = 1} ^h\sum\limits_{\delta = \gamma }^h k^{(3)}_{a\,\gamma \,\delta} \sin\big({x_{\gamma} + x_{\delta}- 2x_{a} -\alpha \big)}.
\end{align*}

We denote with $k^{(2)}_{a\,\beta}$ (resp. $k^{(3)}_{a\,\gamma \,\delta}$) the number of hyperedges of order 2 (resp. 3) shared by nodes in fibre $C_a$ and $C_b$ (resp. $C_a$, $C_\gamma$ and $C_\delta$). By the property of the fibration partition, all nodes $i, j \in C_a$ satisfy $k^{(m)}_{i,\beta} = k^{(m)}_{j,\beta}$ for all $\beta, m$, which justifies the notation $k^{(m)}_{a,\beta}$ for the common value. Since the fibration partition is the coarsest equitable partition, it is not possible to have $k^{(2)}_{i\,\beta} = k^{(2)}_{j\,\beta}$ for all $\beta$ if $i,j$ are not in the same fibre. Since the evolution equation has continuous partial derivatives with respect to every $x_a$, by Picard's theorem of existence and uniqueness of the Cauchy problem, every differential equation has a unique solution associated with the initial condition.    
    Now consider the system of $\lvert N \rvert$ \cref{eq: dynamics_frustration}. Set the same initial condition and natural frequencies for each node in the hypergraph $(\theta_i(0), \omega_i)= (x_0, \omega)$. For each node $i \in C_a$, we consider the equation:
    
    \begin{equation*}
        \dot{\theta}_i = \omega + \sigma_2\sum\limits_{\beta=1}^h k^{(2)}_{i\,\beta} \sin\big({\theta_{\beta} - \theta_{a} -\alpha \big)} + \sigma_3\sum\limits_{\gamma = 1} ^h\sum\limits_{\delta = \gamma }^h k^{(3)}_{i\,\gamma \,\delta} \sin\big({\theta_{\gamma} + \theta_{\delta}- 2\theta_{a} -\alpha \big)}
    \end{equation*}
    The same uniqueness argument applies here. All nodes $i \in C_a$ share the coefficients $k^{(m)}_{a,\beta}$ and thus satisfy the same ODE. Therefore, $\theta_i(t) = x_a(t)$ is the unique solution for each $a \in \{1,\dots\,h\}$, for all $i \in C_a$ and $t \geq 0$, i.e., all nodes in the same fibre evolve identically.
\end{proof}                                
    
\begin{oss}
    When we assign identical initial conditions to all nodes, we eliminate the case of more refined equitable partitions (EP). This is because in finer partitions, nodes that share the same input tree structure could end up in separate clusters created by the EP if the natural frequencies or the initial conditions are different. However, since these nodes follow identical evolution equations and start with the same initial conditions, they would naturally belong to the same cluster under the Kuramoto synchrony partition.
    This means that nodes artificially separated into different clusters would still synchronize with each other. Our objective is to identify the largest possible clusters where all nodes within each cluster achieve synchronization. Therefore, with identical initial conditions, the Kuramoto synchrony partition coincides with the coarsest equitable partition (the fibration), avoiding the degenerate case of global synchronization.
\end{oss}

\begin{oss}\label{obs: degree_uniformity}
    Observe that nodes in the same cluster for the Kuramoto synchrony partition have the same degree for all orders.
    Indeed, let $S_1 \sqcup \dots \sqcup S_j = N$ be the Kuramoto synchrony partition.
    Let $i$ and $j$ belong to the same cluster; by definition, they have the same trajectory and hence the same ODE for all $t \geq 0$.
    We indicate with $k_i^{(m)}$ the generalized degree of order $m$ for node $i$, i.e.\ the number of hyperedges of order $m$ containing the node $i$.
    Their difference equation at $t=0$ is the following:
    \begin{equation*}
    \dot{\theta}_i(0) - \dot{\theta}_j(0) = -\sin{\alpha}\sum_{m=2}^3 \sigma_m (k_i^{(m)}-k_j^{(m)}) = 0.
    \end{equation*}
    Starting from the same initial conditions and natural frequencies for all nodes in the hypergraph, the only way to obtain $\dot{\theta_i}(0) = \dot{\theta_j}(0)$ for all pair of nodes $i,j \in C_a$ and independently of the $\{\sigma_m\}_m$, is to have the same degree sequence, as the frustration parameter belongs to the interval $(0,\pi/2)$.
    Namely:
    \begin{equation*}
        k_i^{(m)} = k_j^{(m)} \,\,\,\forall m, \,\forall i,j \in C_a. 
    \end{equation*}
    Since trajectories remain identical for all $t > 0$, this structural degree condition is necessary for synchrony under our dynamics.
\end{oss}

Kirkland et al.\ in~\cite{kirkland_alphakuramoto_2013} prove that all equitable partitions in graphs are Kuramoto synchrony partitions (called $\alpha$-Kuramoto partitions). However, the reverse does not always hold in their setting. Starting from initial conditions that differ for each partition, graphs can exhibit synchronization among nodes with specific degree relations, even if they do not belong to the same cluster of the equitable partition. How does this extend to hypergraphs where nodes share identical initial conditions and natural frequencies? We now prove that a Kuramoto synchrony partition for our model and homogeneous setting must coincide with a fibration partition, apart from exceptional cases. The result shows that, for generic coupling parameters, any non-fibration synchrony is destroyed instantaneously by a mismatch in incidence patterns. Only nodes with identical higher-order coupling profiles (i.e.\ those in the same fibre) can sustain exact synchrony.

We introduce some lemmas needed to prove the main result.

\begin{lem}\label{lem: sync_all_derivs} Let $\theta(t)$ be a $C^\infty$ solution of~\cref{eq: dynamics_frustration_fibers}.  If for some nodes $i,j$ one has $\theta_i(t)=\theta_j(t)$ for all $t\ge0$, then for $\theta_{ij}(t):=\theta_i(t)-\theta_j(t)$ we have
\begin{equation*}
\theta_{ij}^{(r)}(0)=0\qquad \forall r\ge0.
\end{equation*}
\end{lem}

\begin{proof}
Since $\theta_i(t)=\theta_j(t)$ for all $t\ge0$, we have $\theta_{ij}(t)\equiv 0$ on $[0,\infty)$. Therefore, $\theta_{ij}$ is identically zero in a neighborhood of $t=0$, and hence all its derivatives at $t=0$ vanish.
\end{proof}

\begin{lem}\label{lem: fibre_finite_refinement}
Let $H=(N,E)$ be a hypergraph and consider its associated incidence bipartite graph $B_H=(N\sqcup E, F)$. Let $\{\mathcal P^{(r)}\}_{r\ge0}$ be the sequence of partitions produced by the color refinement \cref{algo: fibers_assign} applied to $B_H$, starting from the initial 2-coloring distinguishing nodes from hyperedges:
\begin{equation*}
\mathcal P^{(0)} = \{N,E\}.
\end{equation*}
This sequence stabilizes to a partition $\mathcal P^{(\infty)}$, which is the fibration partition (equivalently, the coarsest equitable partition of $B_H$ respecting the initial coloring). If nodes $i,j\in N$ belong to different blocks of $\mathcal P^{(\infty)}$, then there exists $r^*\ge1$ such that $i$ and $j$ belong to different blocks of $\mathcal P^{(r^*)}$.
\end{lem}

\begin{proof}
The color refinement algorithm produces a monotone sequence
\begin{equation*}
\mathcal P^{(0)}\succeq \mathcal P^{(1)}\succeq \mathcal P^{(2)}\succeq \cdots
\end{equation*}
where $\mathcal Q\succeq \mathcal R$ means $\mathcal R$ is a refinement of $\mathcal Q$ (every block of $\mathcal R$ is contained in some block of $\mathcal Q$). At each step, two nodes $u,v$ in the same block of $\mathcal P^{(r)}$ are separated in $\mathcal P^{(r+1)}$ if and only if their multisets of neighbor colors in $\mathcal P^{(r)}$ differ.
Since $N\sqcup E$ is finite and each refinement step either strictly refines the partition or leaves it unchanged, the sequence must stabilize after finitely many steps. Let $\mathcal P^{(\infty)}$ denote the stable partition, reached at some step $R<\infty$.
If $i,j\in N$ are in different blocks of $\mathcal P^{(\infty)}$, they cannot be in the same block of $\mathcal P^{(R)}=\mathcal P^{(\infty)}$. Let $r^*$ be the smallest integer such that $i$ and $j$ are in different blocks of $\mathcal P^{(r^*)}$. Then $1\le r^*\le R<\infty$.
\end{proof}

 Using the cluster representation $\theta_\beta$ for the common phase of all nodes in cluster $S_\beta$, we can write for all $i,j \in S_a$: 
   \begin{align*}
        \dot{\theta}_i & = \omega_0 + \sigma_2\sum\limits_{\beta=1}^h k^{(2)}_{i\,\beta} \sin\big({\theta_{\beta} - \theta_{a} -\alpha \big)} + \sigma_3\sum\limits_{\gamma = 1} ^h\sum\limits_{\delta = \gamma }^h k^{(3)}_{i\,\gamma \,\delta} \sin\big({\theta_{\gamma} + \theta_{\delta}- 2\theta_{a} -\alpha \big)}\\
        \dot{\theta}_j & = \omega_0 + \sigma_2\sum\limits_{\beta=1}^h k^{(2)}_{j\,\beta} \sin\big({\theta_{\beta} - \theta_{a} -\alpha \big)} + \sigma_3\sum\limits_{\gamma = 1} ^h\sum\limits_{\delta = \gamma }^h k^{(3)}_{j\,\gamma \,\delta} \sin\big({\theta_{\gamma} + \theta_{\delta}- 2\theta_{a} -\alpha \big)}\
    \end{align*}
    where the double sum with $\delta \ge \gamma$ avoids double-counting three-body interactions, after fixing an order for the clusters $\{S_1, \dots, S_h\}$. Their difference $\theta_{ij} \defeq \theta_i - \theta_j$ vanishes for all $t$. For ease of notation, we omit the dependence of $\theta$ on $t$.    
   \begin{align}\label{eq: diff_phase}
        \dot{\theta}_{ij} = \sigma_2\sum\limits_{\beta=1}^h (k^{(2)}_{i\,\beta}-k^{(2)}_{j\,\beta}) \sin\big({\theta_{\beta} - \theta_{a} -\alpha \big)} + \sigma_3\sum\limits_{\gamma = 1} ^h\sum\limits_{\delta = \gamma }^h (k^{(3)}_{i\,\gamma \,\delta} -k^{(3)}_{j\,\gamma \,\delta})\sin\big({\theta_{\gamma} + \theta_{\delta}- 2\theta_{a} -\alpha \big)}
    \end{align}

\begin{lem}\label{lem: der_ref_diag}
Let $H=(N,E)$ be a hypergraph with $\mathrm{rk}(H)\le 3$ and let $\{\mathcal P^{(r)}\}_{r\ge0}$ be the refinement partitions produced by \cref{algo: fibers_assign} on the incidence bipartite graph $B_H$. Let $i,j\in N$ and let $r^*\ge1$ be the smallest index such that $i$ and $j$ belong to different blocks of $\mathcal P^{(r^*)}$.

Consider the dynamics of \cref{eq: dynamics_frustration_fibers} with identical initial conditions $\Theta(0)=(\theta_0,\ldots,\theta_0)$ and parameters $p=(\alpha,\sigma_2,\sigma_3)$ with $\alpha\in(0,\pi/2)$ and $\sigma_2,\sigma_3\neq 0$. Define $\theta_{ij}(t):=\theta_i(t)-\theta_j(t)$.

Then there exists a nonzero polynomial $P_{ij}^{(r^*)}\in\mathbb Z[\sigma_2,\sigma_3]$ whose coefficients depend only on incidence data depending on the refinement up to depth $r^*$ such that
\begin{equation}\label{eq: diag_poly}
\theta_{ij}^{(r^*)}(0)
= (\sin\alpha)^{a}\,(\cos\alpha)^{b}\,
P_{ij}^{(r^*)}(\sigma_2,\sigma_3),
\end{equation}
for some integers $a\ge1$ and $b\ge0$. In particular, $\theta_{ij}^{(r^*)}(0)\neq0$ for all parameters $(\alpha,\sigma_2,\sigma_3)$ outside the algebraic set $\{P_{ij}^{(r^*)}(\sigma_2,\sigma_3)=0\}$.
\end{lem}

\begin{proof}
We prove by induction on $r\ge1$ that the following statement holds. We want to prove that for every $r\ge1$, there exist integers $a_r\ge1$, $b_r\ge0$ and a polynomial $P_{ij}^{(r)}\in\mathbb Z[\sigma_2,\sigma_3]$ such that
\begin{equation*}
\theta_{ij}^{(r)}(0)
= (\sin\alpha)^{a_r}(\cos\alpha)^{b_r}
P_{ij}^{(r)}(\sigma_2,\sigma_3),
\end{equation*}
where the coefficients of $P_{ij}^{(r)}$ are integer combinations of incidence counts determined by the refinement partition $\mathcal P^{(r)}$.

\noindent\textbf{Base case $r=1$.}
From~\eqref{eq: diff_phase} we have
\begin{equation*}
\dot{\theta}_{ij}
= \sigma_2\!\sum_{\beta}
(k^{(2)}_{i\beta}-k^{(2)}_{j\beta})
\sin(\theta_\beta-\theta_a-\alpha)
+ \sigma_3\!\sum_{\gamma\le\delta}
(k^{(3)}_{i\gamma\delta}-k^{(3)}_{j\gamma\delta})
\sin(\theta_\gamma+\theta_\delta-2\theta_a-\alpha).
\end{equation*}
Under identical initial conditions, all phase differences vanish at $t=0$, so each sine term evaluates to $\sin(-\alpha)=-\sin\alpha$. Therefore
\begin{equation*}
\theta_{ij}^{(1)}(0)
= -\sin\alpha\Bigg[
\sigma_2\sum_{\beta}(k^{(2)}_{i\beta}-k^{(2)}_{j\beta})
+\sigma_3\sum_{\gamma\le\delta}
(k^{(3)}_{i\gamma\delta}-k^{(3)}_{j\gamma\delta})
\Bigg].
\end{equation*}
This is of the required form with $a_1=1$, $b_1=0$, and coefficients depending only on first–step incidence counts, i.e.\ on $\mathcal P^{(1)}$, which splits nodes according to their degrees and hyperedges according to their cardinality.

\noindent\textbf{Induction step.}
Assume the claim holds for some $r\ge1$. By definition,
\begin{equation*}
\theta_{ij}^{(r+1)}(t)
= \frac{d^r}{dt^r}\dot{\theta}_{ij}(t).
\end{equation*}
Each term of $\dot{\theta}_{ij}$ is a linear combination of expressions of the form $\sigma_m\sin(\Phi(t)-\alpha)$, where $\Phi$ is an integer linear combination of cluster phases. Repeated differentiation produces finite sums of products of: (i) trigonometric functions $\sin(\Phi(t)-\alpha)$ or $\cos(\Phi(t)-\alpha)$, and (ii) time derivatives of $\Phi$ of order at most $r$.

Evaluating at $t=0$, because of identical initial conditions, all trigonometric factors reduce to $\pm\sin\alpha$ or $\pm\cos\alpha$, while by the induction hypothesis, each phase derivative
$\Phi^{(s)}(0)$, $1\le s\le r$, is of the form
\begin{equation*}
(\sin\alpha)^{a_s}(\cos\alpha)^{b_s}
(\text{polynomial in }\sigma_2,\sigma_3),
\end{equation*}
with coefficients determined by incidence data up to depth $s$. Multiplying and summing such terms preserves the polynomial structure in $(\sigma_2,\sigma_3)$ and yields an overall factor $(\sin\alpha)^{a_{r+1}}(\cos\alpha)^{b_{r+1}}$ with $a_{r+1}\ge1$. Moreover, the coefficients depend only on incidence counts into blocks of $\mathcal P^{(r)}$, which are exactly the data used to construct $\mathcal P^{(r+1)}$. Thus the claim holds for $r+1$.

By induction, the claim holds for all $r\ge1$. Now let $r^*$ be the smallest refinement depth separating $i$ and $j$. By definition of color refinement, there exists at least one block $C\in\mathcal P^{(r^*-1)}$ for which the incidence multiplicities of $i$ and $j$ into $C$ differ. Such a discrepancy appears in the coefficient structure of $P_{ij}^{(r^*)}$ and does not cancel, implying that $P_{ij}^{(r^*)}\not\equiv0$.

Since $\sin\alpha,\cos\alpha\neq0$ for $\alpha\in(0,\pi/2)$, the \cref{eq: diag_poly} vanishes only when
$P_{ij}^{(r^*)}(\sigma_2,\sigma_3)=0$. Since a nonzero polynomial vanishes only on an algebraic set of Lebesgue measure zero in $\mathbb R^2$, we obtain $\theta_{ij}^{(r^*)}(0)\neq0$ for almost every choice of $(\sigma_2,\sigma_3)$, and hence for almost every choice of parameters $p$ as stated.
\end{proof}

\begin{thm}[Kuramoto synchrony partition $\Rightarrow$ Fibration parition]\label{thm: synch_fibres_identicalIC}
Consider a connected hypergraph $H=(N,E)$ with $\mathrm{rk}(H)\le 3$ and the dynamics of \cref{eq: dynamics_frustration_fibers} with identical natural frequencies $\omega_i=\omega_0$ and identical initial conditions $\theta_i(0)=\theta_0$ for all $i\in N$. Then, for almost every choice of parameters $p=(\alpha,\sigma_2,\sigma_3)$ with $\alpha\in(0,\pi/2)$ and $\sigma_2,\sigma_3\neq 0$, the Kuramoto synchrony partition coincides with the fibration partition of $H$.
\end{thm}
\begin{proof}
Fix parameters $p$ with $\alpha\in(0,\pi/2)$ and $\sigma_2,\sigma_3\neq 0$. Let $S_1\sqcup\cdots\sqcup S_h$ be the Kuramoto synchrony partition of the solution with $\Theta(0)=(\theta_0,\ldots,\theta_0)$.
Take two synchronous nodes $i,j\in S_a$. Then $\theta_i(t)=\theta_j(t)$ for all $t\ge0$, so by \cref{lem: sync_all_derivs} we have $\theta_{ij}^{(r)}(0)=0$ for all $r\ge1$.
Assume for contradiction that $i$ and $j$ lie in different fibres. By \cref{lem: fibre_finite_refinement} there exists a smallest $r^*$ such that $i$ and $j$ are separated at refinement step $r^*$. By \cref{lem: der_ref_diag}, for almost every choice of parameters we have $\theta_{ij}^{(r^*)}(0)\neq 0$, contradicting $\theta_{ij}^{(r^*)}(0)=0$. Therefore, for almost every choice of parameters, synchronized nodes must belong to the same fibre.
\end{proof}

\begin{figure}[!b]
\begin{center}
\begin{minipage}{0.11\textwidth}
    \centering
    
    \vspace{0.3cm}
    \scalebox{0.51}{
\begin{tikzpicture}[scale=0.8]
                \tikzstyle{point}=[circle,draw=black,fill=black,inner sep=0pt,minimum width=4pt,minimum height=4pt]        
                \node (v1) at (-0.7,0.2) {};
                \node (v2) at (0.7,0.2) {};
                \node (v3) at (0,-1) {};
                \node (v4) at (0,-4) {};
                \node (v5) at (-0.7,-5.2) {};
                \node (v6) at (0.7,-5.2) {};
                \node (v7) at (0,-2.5) {};
                \node (v8) at (1.4,-2.5) {};
                \node (v9) at (2.6,-1.8) {};
                \node (v0) at (2.6,-3.2) {};

                \begin{scope}[fill opacity=0.45]

                    \filldraw[black, fill=edgecolor, rounded corners, line width=0.1mm] ([shift={ (-0.35,0.25)}] v1.west) -- ([shift={ (0.35,0.25)}] v2.east) -- ([shift={ (0,-0.35)}] v3.south) -- cycle;

                    \filldraw[black, fill=edgecolor, rounded corners,line width=0.1mm] ([shift={ (-0.35,0)}] v8.west) -- ([shift={ (0.15,0.44)}] v9.east) -- ([shift={ (0.15,-0.44)}] v0.east) -- cycle;
                
                \end{scope}
                    
                \begin{scope}[fill opacity=0.7]

                    \node[circle, fill=c1, draw = black, scale=0.7, label={[label distance=5pt] above left:1}] (v1) at (-0.7,0.2) {};
                    \node[circle, fill=c1, draw = black, scale=0.7, label={[label distance=5pt]above right:2}] (v2) at (0.7,0.2) {};
                    \node[regular polygon, regular polygon sides=3, rotate=180, fill=c6, draw = black, scale=0.6, label={[label distance=5pt]above left:0}] (v3) at (0,-1) {};
                    \node[star, fill=c3, draw = black, scale=0.6, label={[label distance=4pt]left:3}] (v4) at (0,-4) {};
                    \node[regular polygon, regular polygon sides=3, fill=c5, draw = black, scale=0.6pt, label={[label distance=4pt]below left:4}] (v5) at (-0.7,-5.2) {};
                    \node[regular polygon, regular polygon sides=3, fill=c5, draw = black, scale=0.6pt, label={[label distance=4pt]below right:5}] (v6) at (0.7,-5.2) {};
                    \node[rectangle, fill=c4, draw = black, scale=0.8, label={[label distance=4pt]left:6}] (v7) at (0,-2.5) {};
                    \node[regular polygon, regular polygon sides=5, fill=c2, draw = black, scale=0.6pt, label={[label distance=4pt]below left:7}] (v8) at (1.4,-2.5) {};
                    \node[diamond, fill=c7, draw = black, scale=0.6, label={[label distance=5pt]above right:8}] (v9) at (2.6,-1.8) {};
                    \node[diamond, fill=c7, draw = black, scale=0.6, label={[label distance=5pt]below right:9}] (v0) at (2.6,-3.2) {};

                \end{scope}

                \draw[black, line width=0.1mm, thin] (v3) -- (v7);
                \draw[black, line width=0.1mm, thin] (v4) -- (v5);                      
                \draw[black, line width=0.1mm, thin] (v4) -- (v6);               
                \draw[black, line width=0.1mm, thin] (v5) -- (v6);              
                \draw[black, line width=0.1mm, thin] (v7) -- (v4);              
                \draw[black, line width=0.1mm, thin] (v9) -- (v0);            
                \draw[black, line width=0.1mm, thin] (v7) -- (v8);


    \end{tikzpicture}
} 
    
    \vspace{2.57cm} 
    
    \scalebox{0.51}{\begin{tikzpicture}[scale=0.8]
                 \tikzstyle{point}=[circle,draw=black,fill=black,inner sep=0pt,minimum width=4pt,minimum height=4pt]   

                 \begin{scope}[fill opacity=0.7]

                    \node[circle, fill=c1, draw = black, scale=0.7, label={[label distance=5pt] above left:1}] (v1) at (-0.7,0.2) {};
                    \node[circle, fill=c1, draw = black, scale=0.7, label={[label distance=5pt]above right:2}] (v2) at (0.7,0.2) {};
                    \node[regular polygon, regular polygon sides=3, rotate=180, fill=c6, draw = black, scale=0.6, label={[label distance=5pt]above left:0}] (v3) at (0,-1) {};
                    \node[regular polygon, regular polygon sides=3, rotate=180, fill=c6, draw = black, scale=0.6, label={[label distance=4pt]right:3}] (v4) at (0,-4) {};
                    \node[circle, fill=c1, draw = black, scale=0.6, label={[label distance=4pt]below left:4}] (v5) at (-0.7,-5.2) {};
                    \node[circle, fill=c1, draw = black, scale=0.6, label={[label distance=4pt]below right:5}] (v6) at (0.7,-5.2) {};
                    \node[rectangle, fill=c4, draw = black, scale=0.8, label={[label distance=4pt]left:6}] (v7) at (0,-2.5) {};
                    \node[regular polygon, regular polygon sides=3, rotate=180, fill=c6, draw = black, scale=0.6, label={[label distance=4pt]above right:7}] (v8) at (1.4,-2.5) {};
                    \node[circle, fill=c1, draw = black, scale=0.6pt, label={[label distance=5pt]above right:8}] (v9) at (2.6,-1.8) {};
                    \node[circle, fill=c1, draw = black, scale=0.6pt, label={[label distance=5pt]below right:9}] (v0) at (2.6,-3.2) {};

                \end{scope}

                \draw[black, line width=0.1mm, thin] (v3) -- (v7);
                \draw[black, line width=0.1mm, thin] (v4) -- (v5);                      
                \draw[black, line width=0.1mm, thin] (v4) -- (v6);               
                \draw[black, line width=0.1mm, thin] (v5) -- (v6);              
                \draw[black, line width=0.1mm, thin] (v7) -- (v4);              
                \draw[black, line width=0.1mm, thin] (v9) -- (v0);            
                \draw[black, line width=0.1mm, thin] (v7) -- (v8);            
                \draw[black, line width=0.1mm, thin] (v1) -- (v2);              
                \draw[black, line width=0.1mm, thin] (v2) -- (v3);            
                \draw[black, line width=0.1mm, thin] (v1) -- (v3);            
                \draw[black, line width=0.1mm, thin] (v8) -- (v9);            
                \draw[black, line width=0.1mm, thin] (v8) -- (v0);

\end{tikzpicture}}   
    
    \vspace{2.57cm}
    
    \scalebox{0.51}{\begin{tikzpicture}[scale=0.8]
                \tikzstyle{point}=[circle,draw=black,fill=black,inner sep=0pt,minimum width=4pt,minimum height=4pt]   

                \begin{scope}[fill opacity=0.7]

                    \node[circle, fill=c1, draw = black, scale=0.7, label={[label distance=5pt] above left:1}] (v1) at (-0.7,0.2) {};
                    \node[circle, fill=c1, draw = black, scale=0.7, label={[label distance=5pt]above right:2}] (v2) at (0.7,0.2) {};
                    \node[regular polygon, regular polygon sides=3, rotate=180, fill=c6, draw = black, scale=0.6, label={[label distance=5pt]above left:0}] (v3) at (0,-1) {};
                    \node[regular polygon, regular polygon sides=3, rotate=180, fill=c6, draw = black, scale=0.6, label={[label distance=4pt]right:3}] (v4) at (0,-4) {};
                    \node[circle, fill=c1, draw = black, scale=0.6, label={[label distance=4pt]below left:4}] (v5) at (-0.7,-5.2) {};
                    \node[circle, fill=c1, draw = black, scale=0.6, label={[label distance=4pt]below right:5}] (v6) at (0.7,-5.2) {};
                    \node[rectangle, fill=c4, draw = black, scale=0.8, label={[label distance=4pt]left:6}] (v7) at (0,-2.5) {};
                    \node[regular polygon, regular polygon sides=5, fill=c2, draw = black, scale=0.6, label={[label distance=4pt]below left:7}] (v8) at (1.4,-2.5) {};
                    \node[diamond, fill=c7, draw = black, scale=0.6pt, label={[label distance=5pt]above right:8}] (v9) at (2.6,-1.8) {};
                    \node[diamond, fill=c7, draw = black, scale=0.6pt, label={[label distance=5pt]below right:9}] (v0) at (2.6,-3.2) {};

                \end{scope}

                \draw[black, line width=0.1mm, thin] (v3) -- (v7);
                \draw[black, line width=0.1mm, thin] (v4) -- (v5);                      
                \draw[black, line width=0.1mm, thin] (v4) -- (v6);               
                \draw[black, line width=0.1mm, thin] (v5) -- (v6);              
                \draw[black, line width=0.1mm, thin] (v7) -- (v4);             
                \draw[black, line width=0.1mm, thin] (v7) -- (v8);            
                \draw[black, line width=0.1mm, thin] (v1) -- (v2);              
                \draw[black, line width=0.1mm, thin] (v2) -- (v3);            
                \draw[black, line width=0.1mm, thin] (v1) -- (v3);            
                \draw[black, line width=0.1mm, thin] (v8) -- (v9);            
                \draw[black, line width=0.1mm, thin] (v8) -- (v0);

                \path[color = black, bend left=15, line width=0.1mm, thin] (v9) edge node[above] {} (v0);
                \path[color = black, bend left=15, line width=0.1mm, thin] (v0) edge node[above] {} (v9);

\end{tikzpicture}    
              
}   
\end{minipage}
\begin{minipage}{0.8\textwidth}
    \includegraphics[width=\textwidth]{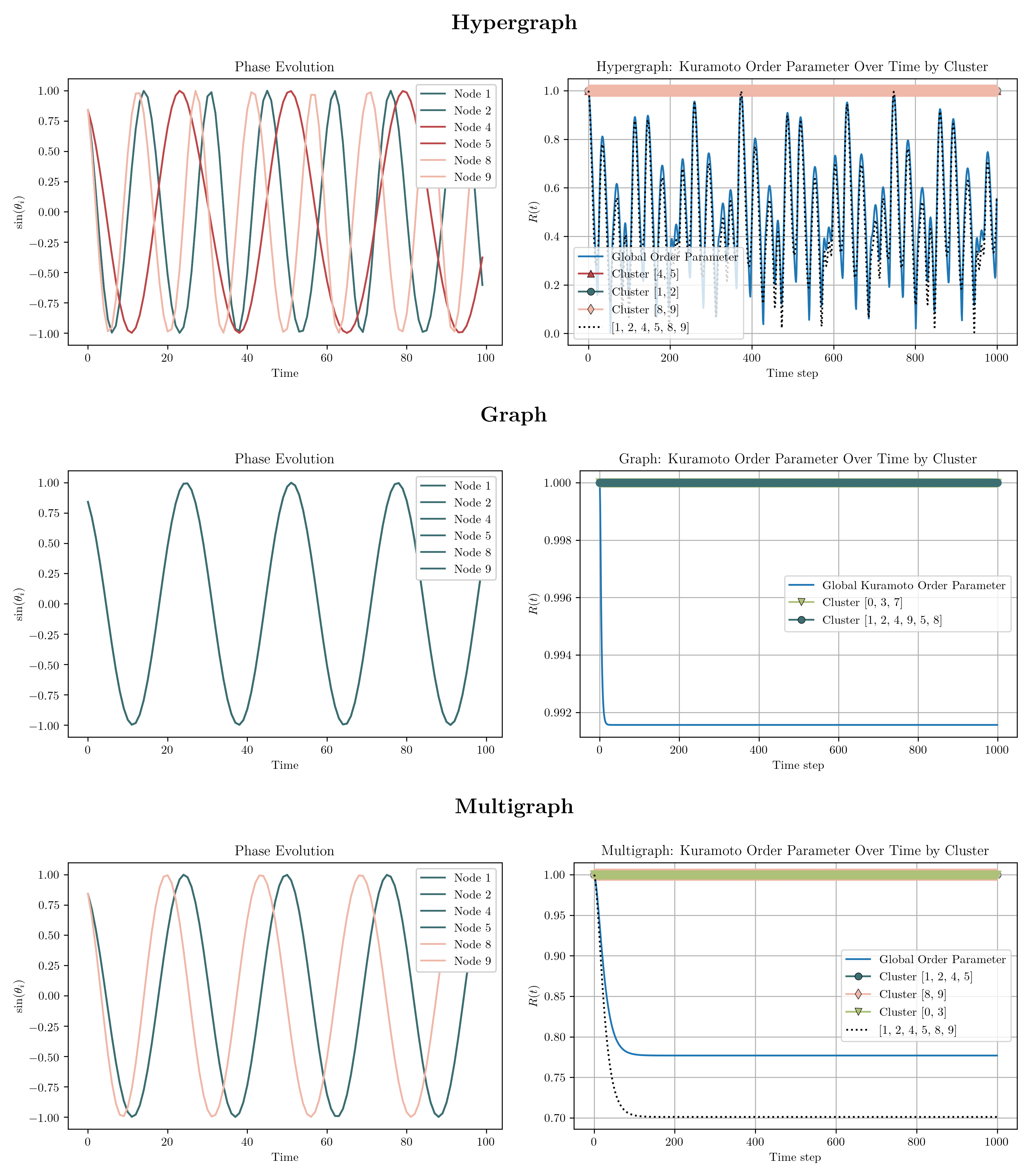}
\end{minipage}
\end{center}
    \caption{\textit{Comparison of the dynamics of the higher-order Kuramoto model with frustration.} The comparison is related to the example shown in \cref{fig: fibers_differences} across three different representations reported on the left of the plots: the hypergraph (top row), its projection in a simple graph (middle row), and its projection in a graph with multiple edges (bottom row). The left plot shows the sine of the phase evolution over the first 100 time steps for the selected nodes $[1,2,4,5,8,9]$ that change fibre depending on the representation, colored according to the fibre they belong to. The right plot reports the local Kuramoto order parameter $R(t)$, measuring phase coherence within each fibre. Colored lines represent intra-fibre synchronization, and the blue thin line shows the global order parameter. For the hypergraph and the multi-order graph (lines 1 and 3), we also report with the dotted black line the measure of coherence computed for the subset of nodes $[1,2,4,5,8,9]$ (which is a fibre only for the simple graph) to illustrate different synchronization patterns across different representations. On the right, the intra-fibre local order parameters are completely overlapped as they are identically 1. This highlights how synchronization patterns can vary significantly depending on the chosen representation. The parameters for the dynamic \cref{eq: dynamics_frustration_fibers} have initial condition and natural frequencies $(\theta_i, \omega_i) = (1,0)$ for all nodes $i$, frustration parameter $\alpha = \pi/6$, and the strength parameters are $(\sigma_2, \sigma_3) = (0.2,0.8)$
    }\label{fig: comparison}
\end{figure}
\begin{oss}
The emergence of synchronization is governed not only by the presence of connections but also by the organization of interactions into fibres, subsets of nodes that share equivalent dynamical roles. These fibres are not intrinsic to the node set alone; they depend on how the system is represented. A hypergraph, its projection to a simple graph, or a multiedge graph can each induce different fibre structures as exemplified in \cref{fig: fibers_differences}. The two previous theorems show that it has a direct impact on which groups of nodes synchronize, starting from the same initial conditions. Only clusters that align with the true fibre structure of the interaction model can be expected to maintain internal coherence. If an incorrect or oversimplified representation is used, synchronization is lost.
This phenomenon is illustrated in \cref{fig: comparison}, where the same dynamical system is represented as a hypergraph (top row), a simple graph (middle), and a multiedge graph (bottom), with reference to \cref{fig: fibers_differences}. Although the parameters and initial conditions are identical, different sets of nodes synchronize in each case, corresponding to the different fibres. 
The set of represented nodes $[1,2,4,5,8,9]$ is split into three different fibres for the hypergraph, a unique fibre for the graph, and two fibres for the multigraph. This structural condition coincides with the dynamical condition of cluster synchronization.
\end{oss}

\subsubsection*{Algorithm for synchronized nodes}
Through the previous theoretical analysis, we established a correspondence between hypergraph fibres and synchronizing nodes within a higher-order Kuramoto model incorporating frustration parameters. 
To empirically verify this relationship, we developed \cref{algo: kuramoto_clusters} that groups nodes based on phase coherence measures. In \cref{obs: degree_uniformity}, we note that nodes with identical initial conditions and natural frequencies will synchronize if they share the same degree at every order, so initially we categorize them as such.
The algorithm then involves temporal sampling by comparing phases pairwise among node combinations within the same split (with the same degree for all orders). For node pairs with aligned phases at all sampled times, we compute the average local order parameter over time to identify nearly unitary values as synchronized clusters.
These pairwise clusters were subsequently merged when sharing common elements, resulting in a partitioning of the nodes. 
To efficiently merge these pairwise synchronized clusters when they share common nodes, we construct an undirected graph where each node represents a Kuramoto oscillator, and edges connect those found to be pairwise synchronized. The connected components of this support graph correspond to the final synchronized clusters.
The experimental results, validated across both real-world datasets and synthetic hypergraphs, reproduced the theoretical predictions of the fibres.
Kuramoto-based clustering methodologies have been extensively investigated in the literature, with notable contributions (e.g.~\cite{bohm_clustering_2010, shao_robust_2013}). Specifically, we can state that when employing a Kuramoto model with identical initial conditions and natural frequencies across all nodes, and non-zero frustration parameters, global synchronization is inhibited. In contrast, local synchronization emerges among nodes that share equivalent inflow information structures.\\

\begin{algorithm}[H]

\SetAlgoLined
\KwIn{Phase history matrix $\Theta \in \mathbb{R}^{T \times \lvert N \rvert }$, node groups $\{N_d\}_d$ by sequence of degree for all hyperedges orders, tolerance $\epsilon > 0$, number of time samples $n$.}
\KwOut{Set of synchronized clusters $C = \{C_1, \dots, C_m\}$.}

\textbf{Step 1: Sample time points.} \\
Randomly select $n$ time indices $\{t_1, \dots, t_n\}$ from $\{1, \dots, T\}$.

\textbf{Step 2: Iterate over degree-based node groups.} \\
\ForEach{group of nodes $V_d = N_d$}{
    Initialize a graph $G_d$ with node set $V_d$.

    \ForEach{pair $(i, j) \in V_d\times V_d$}{
        Extract trajectories $\theta_i(t)$ and $\theta_j(t)$ from $\Theta$.

        \If{$|\theta_i(t_k) - \theta_j(t_k)| < \epsilon$ for all $k = 1, \dots, n$}{
            Compute the average order parameter $\langle R_{ij}\rangle$ over all time.

            \If{$\langle R_{ij}\rangle < \epsilon$}{
                Add edge $(i, j)$ to $G_d$.
            }
        }
    }

    \textbf{Step 3: Extract synchronized clusters.} \\
    Compute connected components $\{C_1^{(d)}, \dots, C_{m_d}^{(d)}\}$ of $G_d$.

    \ForEach{$C_\ell^{(d)}$}{
        Add $C_\ell^{(d)}$ to $C$.
    }
}

\Return{$C$ (maximal synchronized clusters).}
\caption{Extracting synchronized clusters from HO Kuramoto with frustration dynamics.}\label{algo: kuramoto_clusters}
\end{algorithm}

\section{Modifying topology to recover or enforce symmetries}
The role of symmetry in network science has evolved from a purely descriptive framework, centered on automorphisms and equitable partitions, to a tool for designing and repairing networked systems to achieve desired dynamical behavior. Recent advances have emphasized two intertwined goals: the recovery of symmetries in networks distorted by noise, incomplete data, or broken connectivity, and the engineering of symmetries through minimal topological interventions to guide or enhance processes such as synchronization or control. These efforts reveal symmetry not merely as an inherent property, but as a target of intervention, bridging structural analysis and dynamical optimization.

This section provides an overview of recent contributions in this paradigm. Each approach addresses the challenge of how to modify a network to reflect or enforce a desired symmetry, whether derived from observed dynamics, known structural constraints, or theoretical design goals.
Introducing a rigorous framework to approach the problem, Boldi et al.~\cite{boldi_quasifibrations_2022} define the notion of quasifibrations as a way to recover approximate symmetries in biological networks by identifying minimal modifications to the topology that restore an underlying fibration-like structure. Similarly, Leifer et al.~\cite{leifer_symmetrydriven_2022} develop a symmetry-driven network reconstruction method that modifies edges to satisfy local balance constraints based on pseudocoloring, offering an approach to repair corrupted networks and obtain the optimized number of clusters. In a different direction, Gambuzza et al.~\cite{gambuzza_controlling_2021} focus on enforcing desired symmetries to achieve clustered dynamics, proposing targeted topological interventions based on external equitable partitions that guarantee stable cluster synchronization. Other works adopt generative or learning-based approaches: Klickstein and Sorrentino~\cite{klickstein_generating_2018} construct synthetic networks with predefined automorphism groups to serve as symmetry benchmarks, while Zou et al.~\cite{zou_rewiring_2024} use reinforcement learning with graph neural networks to rewire network links and achieve symmetry patterns that support cluster synchronization. Closing the loop between theory and application, Zou et al.~\cite{zou_control_2025} show how structural symmetry can be leveraged in physical systems to control chaotic synchronization, proving that altering a network to match dynamical symmetry is both theoretically justified and experimentally feasible. Together, these contributions (see the comparative \cref{tab: symmetry_methods_comparison}) highlight a growing perspective where symmetries are not only analysed but also shaped or repaired, providing a natural foundation for extending these ideas to hypergraphs, where higher-order symmetries can be studied through generalized fibrations. The field of hypergraph modifications for symmetry analysis raises many questions and provides directions for future research. 

\begin{table}[htbp]
    \centering    
    \begin{booktabs}{
        colspec = {X[0.9,l]X[0.9,l]X[0.9,l]X[0.9,l]},
        rowsep = 3pt,        
        row{1} = {font=\bfseries}
    }    
        \toprule
        Work & Target problem & Mathematical object & Main technique \\
        \midrule
        Boldi et al.\ (2022)~\cite{boldi_quasifibrations_2022} & Recover latent symmetries from incomplete graphs & Quasifibration & Greedy and local isomorphism checks \\
        
        Gambuzza et al.\ (2021)~\cite{gambuzza_controlling_2021} & Impose a target cluster pattern & Balanced Laplacian perturbation & Mixed-integer and convex optimisation \\
        
        Klickstein et al.\ (2018)~\cite{klickstein_generating_2018} & Produce test graphs with chosen symmetry & Automorphisms and equitable partitions & Constructive group-theoretic generator \\
        
        Leifer et al.\ (2022)~\cite{leifer_symmetrydriven_2022} & Repair noisy networks via pseudosymmetries & Pseudobalanced colouring & Integer programming \\
        
        Zou et al.\ (2024)~\cite{zou_rewiring_2024}& Learn rewiring policy for achievable cluster sync & Graph Convolution Network & Reinforcement learning \\
        \bottomrule
    \end{booktabs}
    
    \caption{
    \label{tab: symmetry_methods_comparison}\textit{Comparison of some recent approaches for recovering or enforcing network symmetries.}}
\end{table}

In what follows, we make an initial contribution by tackling challenges arising from realistic data scenarios and computational constraints: reconstructing latent symmetries from partial observations and eliminating structural redundancy, maintaining equivalence relationships for information flows. \cref{fig: modifications_examples} recaps with three simple examples the problems.

\begin{figure}[ht]
    \centering
    \scalebox{0.8}{\begin{tikzpicture}[scale=1]

\coordinate (O1) at (0,0);
\coordinate (O2) at (6,0);
\coordinate (O3) at (12,0);

\node[align=center] at ($(O1)+(0,3.3)$) {\textbf{1. Sparsification}\\\textit{Fibre-preserving reduction}};

\node[star, fill=hp4!60, draw=black, minimum size=6pt, scale=0.7, label={[label distance=7pt]left:A}] (a1) at ($(O1)+(-0.7,0.2)$) {};
\node[star, fill=hp4!60, draw=black, minimum size=6pt, scale=0.7, label={[label distance=7pt]right:B}] (b1) at ($(O1)+(0.7,0.2)$) {};
\node[circle, fill=hp1!60, draw=black, minimum size=6pt, scale=0.9, label={[label distance=7pt]below:C}] (c1) at ($(O1)+(0,-1)$) {};
\node[rectangle, fill=hp3!80, draw=black, minimum size=6pt, label={[label distance=7pt]above:D}] (d1) at ($(O1)+(0,1.4)$) {};

\begin{scope}[fill opacity=0.3]
\filldraw[black, fill=edgecolor, thin, rounded corners] 
    ([shift={(-0.35,0.25)}] a1.west) --
    ([shift={(0.35,0.25)}] b1.east) --
    ([shift={(0,-0.35)}] c1.south) -- cycle;
\end{scope}

\draw[dashed, line width=0.5mm] (c1) -- (a1);
\draw[dashed, line width=0.5mm] (c1) -- (b1);
\draw (d1) -- (a1);
\draw (d1) -- (b1);

\node[align=left, font=\footnotesize] at ($(O1)+(0,-2.5)$) {
    Example\\
    Original fibres: $\{A,B\},\{C\},\{D\}$\\
    Target fibres: $\{A,B\},\{C\},\{D\}$\\
    Removed hyperedges: $(A,C), (C,B)$
};

\node[align=center] at ($(O2)+(0,3.3)$) {\textbf{2. Redundancy Injection}\\\textit{Non-invasive reinforcement}};

\node[star, fill=hp4!60, draw=black, minimum size=6pt, scale=0.7, label={[label distance=7pt]left:A}] (a2) at ($(O2)+(-0.7,0.2)$) {};
\node[star, fill=hp4!60, draw=black, minimum size=6pt, scale=0.7, label={[label distance=7pt]right:B}] (b2) at ($(O2)+(0.7,0.2)$) {};
\node[circle, fill=hp1!60, draw=black, minimum size=6pt, scale=0.9, label={[label distance=7pt]below:C}] (c2) at ($(O2)+(0,-1)$) {};
\node[rectangle, fill=hp3!80, draw=black, minimum size=6pt, label={[label distance=7pt]above:D}] (d2) at ($(O2)+(0,1.4)$) {};

\begin{scope}[fill opacity=0.3]
\filldraw[black, fill=edgecolor, thin, rounded corners] 
    ([shift={(-0.35,0.25)}] a2.west) --
    ([shift={(0.35,0.25)}] b2.east) --
    ([shift={(0,-0.35)}] c2.south) -- cycle;
\end{scope}

\draw[black, line width=0.5mm] (c2) -- (a2);
\draw[black, line width=0.5mm] (c2) -- (b2);
\draw[black, line width=0.5mm] (a2) -- (b2);
\draw (d2) -- (a2);
\draw (d2) -- (b2);

\node[align=left, font=\footnotesize] at ($(O2)+(0,-2.5)$) {
    Example\\
    Original fibres: $\{A,B\},\{C\},\{D\}$\\
    Target fibres: $\{A,B\},\{C\},\{D\}$\\
    Added hyperedges: $(A,C), (C,B), (A,B)$
};

\node[align=center] at ($(O3)+(0,3.3)$) {\textbf{3. Topology Correction}\\\textit{Fibration alignment}};

\node[star, fill=hp4!60, draw=black, minimum size=6pt, scale=0.7, label={[label distance=7pt]left:A}] (a3) at ($(O3)+(-0.7,0.2)$) {};
\node[star, fill=hp4!60, draw=black, minimum size=6pt, scale=1, label={[label distance=7pt]right:B}] (b3) at ($(O3)+(0.7,0.2)$) {};
\node[circle, fill=hp1!60, draw=black, minimum size=6pt, scale=0.9, label={[label distance=7pt]below:C}] (c3) at ($(O3)+(0,-1)$) {};
\node[rectangle, fill=hp3!80, draw=black, minimum size=6pt, label={[label distance=7pt]above:D}] (d3) at ($(O3)+(0,1.4)$) {};

\begin{scope}[fill opacity=0.3]
\filldraw[black, fill=edgecolor, line width=0.2mm, rounded corners] 
    ([shift={(-0.35,0.25)}] a3.west) --
    ([shift={(0.35,0.25)}] b3.east) --
    ([shift={(0,-0.35)}] c3.south) -- cycle;
\end{scope}

\draw[line width=0.2mm] (c3) -- (a3);
\draw[line width=0.2mm] (a3) -- (b3);
\draw[line width=0.2mm] (d3) -- (a3);
\draw[line width=0.2mm] (d3) -- (b3);
\draw[line width=0.5mm] (c3) -- (b3);

\node[diamond, fill=hp5!60, draw=black, minimum size=6pt, scale=0.5] at ($(O3)+(0.7,0.2)$) {};

\node[align=left, font=\footnotesize] at ($(O3)+(0,-2.5)$) {
    Example\\
    Original fibres: $\{A\},\{B\},\{C\},\{D\}$\\
    Target fibres: $\{A,B\},\{C\},\{D\}$\\
    Added hyperedges: $(C,B)$
};

\end{tikzpicture}} 
    \caption{\textit{Visualization of considered hypergraph modification cases.} Examples of hypergraph modifications aimed at: (1) reducing or (2) increasing the redundancy of informative hyperedges while preserving the original fibre partition of the nodes. Case (3) illustrates the addition of hyperedges to achieve a target fibre partition. Dashed hyperedges represent those removed by the algorithm, while thick hyperedges indicate additions made to enforce a specific input configuration on the nodes.}\label{fig: modifications_examples}
\end{figure}

\subsection{Sparsification of a hypergraph preserving fibres}

We address a problem related to higher-order interactions in network modification for the recovery or enforcement of symmetry. As already observed, hypergraphs represent systems with multi-node interactions, such as social groups, biochemical reactions, or communication networks. These often contain redundant hyperedges that contribute little to distinguish nodes having different information flows. Removing such redundancies simplifies the hypergraph, improves interpretability, and speeds up analysis and simulation. We seek a process that preserves connectivity, avoids adding new connections, and maintains fibre structures. Starting from a hypergraph, we aim to find a connected hypergraph having the same nodes as the original and with the minimum number of hyperedges needed to preserve the partition of the nodes induced by the fibration symmetry.

\paragraph*{Algorithmic approach}
When partitioning the nodes in fibres with \cref{algo: fibers_assign}, a color is also assigned to each hyperedge, based on the fibres of the nodes belonging to it. Then, \cref{algo: sparse} exploits this coloring scheme and applies a greedy, iterative process to remove these hyperedge classes one by one, accepting removals only when they preserve the fibre partition and connectivity. To explore the solution space, it tests multiple random permutations of the order in which hyperedge classes are removed, retaining the sparsification that yields the smallest hyperedge set while meeting constraints. Additionally, in a variant of the algorithm, it is possible to specify a set of hyperedges that must not be removed, either because their presence is experimentally validated or because they are known to be structurally critical. In this case, the sparsification procedure is constrained to preserve those hyperedges.
While the greedy approach is computationally efficient for large hypergraphs, it does not guarantee a globally optimal solution; hence, the minimal possible hyperedge set preserving fibres and connectivity may not be found. This limitation arises because decisions are made sequentially without backtracking or exhaustive search, potentially missing better combinations of removals. However, the method combines feasibility and structural consistency, providing a sparsification that respects fibration symmetries and connectivity, which is sufficient for many applications.

\begin{algorithm}[H]
\SetAlgoLined
\KwIn{Hypergraph $H = (N, E)$ and a node partition $C = \{C_i\}$ representing fibres.}
\KwOut{A reduced hypergraph $H' = (N, E')$ with $\lvert E'\rvert \leq \lvert E\rvert$ preserving both fibre structure and connectivity.}

\textbf{Step 1: Group hyperedges by structural color.} \\
Group hyperedges that share the same color obtained with \cref{algo: fibers_assign}.

\textbf{Step 2: Generate candidate color removal orders.} \\
Define several orders in which to attempt removing hyperedge color groups.

\textbf{Step 3: Iteratively attempt to remove hyperedges.} \\
\ForEach{color group in the selected order}{
    Temporarily remove all hyperedges of this color group.\\
    \If{the resulting hypergraph is connected and the fibre partition preserved}{
        Accept the removal permanently.
    }
    \Else{
        Roll back the removal.
    }
}

\textbf{Step 4: Return the sparsified hypergraph.} \\
Among all the sets of new hyperedges computed in Step 3 in correspondence of different color orders, keep the smallest $E'$.
\caption{Greedy sparsification of a hypergraph while preserving fibre partitions and connectivity.}\label{algo: sparse}
\end{algorithm}

\subsection{Modifying hypergraphs to match target fibres}
In this work, we have observed that nodes belonging to the same fibre often exhibit synchronization or tightly correlated behavior, as predicted by the underlying local symmetries of the hypergraph. However, in real-world or biological systems, the observed topology may be incomplete or corrupted due to measurement noise, data loss, or partial observation. As a result, the hypergraph topology might not reflect the true interaction structure responsible for the observed dynamical patterns. When we observe (or can deduce) that specific nodes exhibit synchronized or identical behavior, we can leverage this information to reconstruct a likely hypergraph structure. This reconstructed hypergraph would reproduce the observed pattern of synchronized node groups while maintaining maximum similarity to the original network's structure and dynamics.

\paragraph*{Algorithmic approach}
The approach we propose in \cref{algo: refine_and_sparsify} is greedy and heuristic, relying on iterative adjustments. 
Initially, we iteratively compare the actual fibres and the target fibres after adding new hyperedges. Then we use a version of the pruning \cref{algo: sparse} to reduce the number of hyperedges while keeping the original hyperedges in $E$: in this way, we avoid removing original hyperedges.
Since the process depends on incremental structural corrections, convergence is not guaranteed, especially for complex or conflicting target partitions.

\begin{algorithm}[H]
\SetAlgoLined
\KwIn{Hypergraph $H = (N, E)$ and a target fibre partition $C = \{C_i\}$.}
\KwOut{A hypergraph $H' = (N, E')$ with fibre structure $C$.}

\textbf{Step 1: Initialization.} \\
Compute P, the fibre partition of H.

\textbf{Step 2: Refine the fibre structure.} \\
\While{$P$ differs from $C$ and iteration limit not reached}{
    Compute current fibre partition $P$. \\
    \ForEach{cluster in $P \setminus C$}{
        Add minimal differences to split the cluster.
    }
    \ForEach{cluster in $C \setminus P$}{
        Add minimal common connectivity to merge the cluster.
    }
}

\textbf{Step 3: Sparsify the hypergraph.} \\
Apply \cref{algo: sparse} to remove redundant hyperedges not removing original hyperedges in $E$ while preserving $C$ and connectivity.

\textbf{Step 4: Return final result.} \\
Output the resulting hypergraph $(N, E')$.

\caption{Modification of a hypergraph to reach target fibre partition.}\label{algo: refine_and_sparsify}
\end{algorithm}

\subsection{Adding redundant hyperedges without altering fibres}\label{subsec: red}
While minimizing the number of hyperedges is often desirable, strategically adding redundancy can reinforce symmetries, preserve fibre structure, and improve the functional and dynamical properties of the hypergraph without altering the essential partition, enhancing robustness to perturbations and enabling more regular, compressible representations that facilitate analysis and model reduction.

\paragraph*{Algorithmic approach}
The method described in \cref{algo: add_structured_hyperedges} adds new hyperedges in a way that reinforces a given fibre partition $C$ while preserving it. The procedure is structural and randomized: it samples combinations of clusters and generates new hyperedges by connecting nodes from those clusters. To ensure consistency, the number of new hyperedges added at each step is the least common multiple of the sizes of the selected clusters, and nodes are paired cyclically. In this way, we do not disrupt the connection coherence intra-cluster. 
In detail, at each iteration, the algorithm selects a target hyperedge order $r$ from a pre-shuffled list of integers. This shuffling ensures a fair exploration of interaction orders, avoiding any systematic bias toward lower or higher orders.
To construct a candidate set of hyperedges, the algorithm randomly picks $r$ distinct clusters from the partition $C$, where $r$ is the chosen order. Importantly, it may not be enough to add a single hyperedge: to preserve the fibre structure, all nodes within each selected cluster must receive identical new connection patterns. This guarantees that their structural role in the hypergraph remains symmetric.
To enforce this uniformity, the algorithm computes the least common multiple $L$ of the sizes of the selected clusters. This value determines the number of hyperedges to generate. For each cluster, its nodes are sampled without replacement and then repeated cyclically until a list of length $L$ is formed. These repeated lists are then zipped together with one node from each list per position, to produce exactly $L$ hyperedges. Each of these new hyperedges connects one node from each selected cluster, ensuring that every node in those clusters appears the same number of times across the batch. The result is a structurally consistent set of interactions that respects the symmetries encoded by $C$.
Each batch of candidate hyperedges is accepted only if the resulting hypergraph preserves the original fibre partition. Indeed, even if the addition of hyperedges to maintain the same kind and number of connections intra-fibre is preserved, some clusters may collapse and unite into a single, larger fibre. This procedure allows the hypergraph to gain structural redundancy while strictly maintaining the predefined node grouping.

\begin{algorithm}[H]
\SetAlgoLined
\KwIn{Hypergraph $H = (N, E)$ with fibre partition $C = \{C_i\}$, maximum number of hyperedges to add $K$, maximum number of iterations $T$.}
\KwOut{A hypergraph $H' = (N, E')$ with $\lvert E'\rvert \geq \lvert E\rvert$ preserving both fibre structure $C$.}

\textbf{Step 1: Initialization.} \\
Let $E' \gets E$ be the set of current hyperedges. \\
Let $d = rk(H)$. \\
Initialize a counter for added hyperedges. \\
Let $T$ be the total number of trials (iteration limit). \\
Define a random sequence over orders  $2, \dots, d$.

\textbf{Step 2: Iteratively attempt to add new hyperedges.} \\
\While{number of added hyperedges $< K$ and trials left}{
    Pick order $r$ from the random sequence. \\
    Randomly select $r$ distinct clusters from $C$. \\
    \For{$i \gets 1$ to $3$ (retry limit)}{
        Compute the least common multiple $L$ of the sizes of the selected clusters. \\
        For each cluster $C_i$, sample its nodes without replacement and repeat them cyclically to reach length $L$. \\
        Construct $L$ candidate hyperedges by zipping together one node from each repeated cluster list. \\
        \If{none of the candidate hyperedges already exist in $E'$}{
            Temporarily add them to form $E''$. \\
            \If{fibre partition of $H'' = (N, E'')$ equals $C$}{
                Commit $E' \gets E''$. \\
                Update the counter of added hyperedges. \\
                \textbf{break} retry loop.
            }
        }
    }
}

\textbf{Step 3: Return final result.} \\
Output the resulting hypergraph $H' = (N, E')$.

\caption{Addition of redundant hyperedges preserving a target fibre partition.}\label{algo: add_structured_hyperedges}
\end{algorithm}

\subsection{Discussion}
In this work, we have explored the relationship between local structural symmetries in undirected hypergraphs and the emergence of synchronization patterns in dynamical systems. We have deepened the understanding of fibration symmetry for complex systems with higher-order interactions, using the frustrated Kuramoto model as an example to study the dependence of topology on the emergence of cluster synchronization patterns. However, our theoretical framework relies on a critical assumption that significantly limits its direct applicability to real-world processes: the requirement of identical initial conditions and natural frequencies across all nodes. This assumption, while mathematically necessary for our theoretical results, is unrealistic in natural systems. Real networks invariably exhibit heterogeneity in both initial states and intrinsic dynamics, making our exact synchrony conditions often unattainable. Despite these limitations, the duality between structural fibration and dynamical synchronization revealed through the Kuramoto framework provides valuable theoretical insights. It offers alternatives for algorithmic applications in clustering and community detection. The Kuramoto model's ability to naturally organize nodes according to underlying structural symmetries suggests potential when considering frustration parameters to inhibit global synchronization. For practical relevance, we suggest greedy hypergraph compression methods that keep fibration compatibility and employ topology modification strategies to achieve desired node partitions, useful for handling incomplete or noisy data.
Future work could explore generalizations of the framework introduced here, including the development of enhanced algorithms beyond the greedy approaches presented here for topology modification problems. Such extensions help capture more nuanced forms of partial synchronization in systems where the exact fibre structure is absent.

\section{Methods}
All simulations carried out throughout the paper, unless otherwise stated, solve the Cauchy problem by computing trajectories using the method (\texttt{LSODA}) from the Python library \href{https://docs.scipy.org/doc/scipy/reference/generated/scipy.integrate.solve_ivp.html}{SciPy} with step size 0.1. Synthetic hypergraphs are generated using \texttt{random hypergraph} from the \texttt{hypergraphx.generation.random} module from the Python library \href{https://hypergraphx.readthedocs.io/en/master/index.html}{Hypergraphx} by specifying the number of nodes and hyperedges of orders 2 and 3, then extracting the largest connected component.

\section{Data availability}
The datasets used in \cref{sub: real} are from \href{https://github.com/xgi-org/xgi-data}{XGI dataset}~\cite{landry_xgi_2023} and the MAG-10 dataset is from the preprocessed files in \url{https://github.com/TheoryInPractice/overlapping-ecc/tree/master/data/MAG-10} ~\cite{amburgilya_clustering_2020, sinha_overview_2015}.

\section{Code availability}\label{sec: code}
The code used for the analysis in the paper is organized in \url{https://github.com/margheritaberte/Fibration_symmetries_cluster_synchronization}.
The repository provides functions to compute the fibre partition (i.e., the coarsest equitable partition or minimal balanced coloring) for nodes of graphs, multigraphs, and hypergraphs, along with example notebooks that illustrate, visualize, and apply the algorithms discussed in the paper.

\clearpage
\bibliography{Symmetries}

\section{Competing interests}
The authors declare no competing interests.

\newpage
\appendix
\section{Supplementary material}
\subsection{Categorical equivalence}\label{sec: appendix_cat}

In this section, we prove the categorical equivalence between bipartite graphs and hypergraphs. Given a category $C$, we indicate its objects with $Ob(C)$ and its morphisms as $Mor(C)$ or $Mor_C(A, B)$ if we need to specify source objects $A$ and target objects $B$.

\begin{defn}[$\mathbf{Hyper}$]\label{def: cat_hypg}
    The category $\mathbf{Hyper}$ has as objects the hypergraphs as defined in \cref{def: hypg}. A morphism $f \in \mathrm{Mor}(\mathbf{Hyper})$ between two hypergraphs $H_1=(N_1,E_1)$ and $H_2=(N_2,E_2)$ is a pair of maps $f=(f_N,f_E): H_1\to H_2$, where $f_N:N_1\to N_2$ and $f_E:E_1\to E_2$, that satisfy: 
    $$v\in e \;\Longrightarrow\; f_N(v)\in f_E(e)$$ 
    for all $v\in N_1$ and $e\in E_1$. Hence, morphisms are maps between hypergraphs that send nodes to nodes and hyperedges to hyperedges, preserving the incidence relation.
\end{defn}

\begin{defn}[$\mathbf{BiparGraph}$]\label{def: cat_bip}
    The category $\mathbf{BiparGraph}$ has as objects the bipartite graphs as defined in \cref{def: bip_graph}. A morphism between two bipartite graphs $G_A = (V_{A_1} \sqcup V_{A_2}, E_A)$ and  $G_B = (V_{B_1} \sqcup V_{B_2}, E_B)$ is a triple of maps $f = (f_1, f_2, f_E)$ where:
    \begin{itemize}
        \item $f_1: V_{A_1} \rightarrow V_{B_1}$
        \item $f_2: V_{A_2} \rightarrow V_{B_2}$
        \item $f_E: E_A \rightarrow E_B$ defined by $f_E((v,u)) = (f_1(v), f_2(u))$
    \end{itemize}
    and $(v,u)\in E_A \;\Rightarrow\; (f_1(v),f_2(u))\in E_B$. Hence, morphisms are maps between bipartite graphs that send nodes of the first (resp. second) layer to nodes of the first (resp. second) layer, preserving adjacency.
\end{defn}

\begin{defn}[Functor]
    Given two categories $C, D$, a functor $F: C \rightarrow D$ is a map that:
    \begin{itemize}
        \item $\forall x \in \mathrm{Ob}(C) \Rightarrow F\,x \in \mathrm{Ob}(D)$.
        \item If $f \in \mathrm{Mor}(C)$, $f: X \rightarrow Y$ with $X, Y \in \mathrm{Ob}(C)$, then $F\,f : F\,X \rightarrow F\,Y$ is a morphism of $D$ such that:
        \begin{itemize}
            \item $F\,\mathrm{id}_{X} = \mathrm{id}_{F\,X} \quad\forall X \in \mathrm{Ob}(C)$,
            \item $F(f \circ g) = F\,f \circ F\,g \quad \forall f: X \rightarrow Y, \,g: Y \rightarrow Z$ with $X, Y, Z \in \mathrm{Ob}(C)$.
        \end{itemize} 
    \end{itemize}
\end{defn}

\begin{defn}[Categorical equivalence]\label{def: cat_eq}
    Two categories $C, D$ are equivalent if there exists a functor $F: C \rightarrow D$ such that:
    \begin{itemize}
        \item $F$ is full: $\forall c_1, c_2 \in \mathrm{Ob}(C)$, the map $\mathrm{Mor}_C(c_1, c_2) \longrightarrow \mathrm{Mor}_D(F\,c_1, F\,c_2)$ induced by $F$ is surjective.
        \item $F$ is faithful: $\forall c_1, c_2 \in \mathrm{Ob}(C)$, the map $\mathrm{Mor}_C(c_1, c_2) \longrightarrow \mathrm{Mor}_D(F\,c_1, F\,c_2)$ induced by $F$ is injective.
        \item $F$ is dense (essentially surjective): $\forall d \in \mathrm{Ob}(D)$, $\exists c \in \mathrm{Ob}(C)$ such that $d \cong F\,c$ (i.e., $d$ is isomorphic to $F\,c$).
    \end{itemize}
\end{defn}

\begin{thm}
    The categories $\mathbf{Hyper}$ and $\mathbf{BiparGraph}$ are equivalent.
\end{thm}

\begin{proof}
    We construct a functor $F: \mathbf{Hyper} \rightarrow \mathbf{BiparGraph}$ and prove it satisfies the conditions in \cref{def: cat_eq}.
    
    \textbf{Functor definition:}
    \begin{description}
        \item[On objects] Given a hypergraph $H = (N, E)$, define the bipartite graph
        $$F\,H \coloneqq (N \sqcup E, A_H)$$
        where $A_H = \{(v, e) \mid v \in N, e \in E, v \in e\}$.
        
        \item[On morphisms] Given a hypergraph morphism $g=(g_N,g_E):H_1\to H_2$, define 
        $$F\,g \coloneqq (g_N, g_E, g_A)$$
        where $g_A((v,e)) = (g_N(v), g_E(e))$ for $(v,e) \in A_{H_1}$.
    \end{description}
    
    \textbf{Verification that $F$ is a functor:}
    \begin{description}
        \item[Preserves identities] For any hypergraph $H$, $F(\mathrm{id}_H) = (\mathrm{id}_N, \mathrm{id}_E, \mathrm{id}_{A_H}) = \mathrm{id}_{F\,H}$.
        \item[Preserves composition] For morphisms $g_1: H_1 \to H_2$ and $g_2: H_2 \to H_3$, we have
        $$F(g_2 \circ g_1) = F((g_{2,N} \circ g_{1,N}, g_{2,E} \circ g_{1,E})) = (g_{2,N} \circ g_{1,N}, g_{2,E} \circ g_{1,E}, (g_{2,A} \circ g_{1,A})) = F\,g_2 \circ F\,g_1.$$
    \end{description}
    
    Now we prove that $F$ is full, faithful, and essentially surjective.

    \begin{description}
        \item[Full] Take two hypergraphs $H_1 = (N_1, E_1)$ and $H_2 = (N_2, E_2)$, so $F\,H_1 = (N_1 \sqcup E_1, A_1)$ and $F\,H_2 = (N_2 \sqcup E_2, A_2)$.
        
        Let $f = (f_1, f_2, f_A): F\,H_1 \rightarrow F\,H_2$ be a morphism in $\mathbf{BiparGraph}$.
        Define $g = (f_1, f_2): H_1 \rightarrow H_2$. We need to verify that $g$ is a morphism in $\mathbf{Hyper}$.
        
        For any $v \in N_1$ and $e \in E_1$ with $v \in e$, we have $(v,e) \in A_1$. Since $f$ is a bipartite graph morphism, $(f_1(v), f_2(e)) \in A_2$, which means $f_1(v) \in f_2(e)$. Thus $g$ preserves incidence and is a hypergraph morphism with $F\,g = f$.

        \item[Faithful] Consider two hypergraph morphisms $g_1, g_2: H_1 \rightarrow H_2$.
        Suppose $F\,g_1 = F\,g_2$. Then:
        \begin{align*}
            F\,g_1 &= (g_{1,N}, g_{1,E}, g_{1,A}) = (g_{2,N}, g_{2,E}, g_{2,A}) = F\,g_2
        \end{align*}
        This implies $g_{1,N} = g_{2,N}$ and $g_{1,E} = g_{2,E}$, hence $g_1 = g_2$.
        
        \item[Dense] Given a bipartite graph $G = (V_1 \sqcup V_2, A)$, define the hypergraph $H = (N, E)$ where:
        \begin{itemize}
            \item $N = V_1$
            \item $E = V_2$, and for each $e \in V_2$, define the corresponding hyperedge as the set $\{v \in V_1 \mid (v, e) \in A\}$
        \end{itemize}
        
        Then $F\,H = (V_1 \sqcup V_2, \{(v, e) \mid v \in V_1, e \in V_2, v \in e\}) = (V_1 \sqcup V_2, A) = G$.
        
        Thus $G = F\,H$.
    \end{description}
    
    Since $F$ is full, faithful, and dense, the categories $\mathbf{Hyper}$ and $\mathbf{BiparGraph}$ are equivalent.
\end{proof}

\end{document}